\documentclass[11pt]{article}

\RequirePackage[OT1]{fontenc}
\RequirePackage{amsthm,amsmath}
\RequirePackage[numbers]{natbib}
\RequirePackage[colorlinks,citecolor=blue,urlcolor=blue]{hyperref}
\usepackage[utf8]{inputenc}
\usepackage[dvipsnames]{xcolor}
\usepackage{mathtools,amssymb, dsfont, mathrsfs, amsthm, graphicx, pgf, float, cleveref,csquotes,bbm,bm}

\usepackage[shortlabels]{enumitem}
\usepackage{xargs, array,afterpage,booktabs}
\usepackage{caption}
\usepackage{subcaption}
\usepackage{pgfplots}
\usetikzlibrary{arrows}
\usepackage{multirow}

\newcommand{\IA}{\mathbb{A}}
\newcommand{\ia}{\mathbbm{a}}
\newcommand{\IE}{\mathbb{E}}
\newcommand{\Ii}{\mathbbm{1}}
\newcommand{\IR}{\mathbb{R}}
\newcommand{\IN}{\mathbb{N}}
\newcommand{\IZ}{\mathbb{Z}}
\newcommand{\IP}{\mathbb{P}}
\newcommand{\dif}{\,\text{d}}
\newcommand{\ssG}{\mathcal{G}}
\newcommand{\ssF}{\mathcal{F}}
\newcommand{\Epo}{\mathcal{E}}
\newcommand{\reluD}{\mathcal{R}_D}
\newcommand{\reluG}{\mathcal{R}_G}
\newcommand{\reff}[1]{(\ref{#1})}

\newcommand{\fg}{\mathcal{R}_{fg}}
\newcommand{\wnc}{W_{1,n}^c}
\newcommand{\wn}{W_{1,n}}
\newcommand{\minitab}[2][l]{\begin{tabular}{#1}#2\end{tabular}}

\newcommand{\relu}{\mathcal{R}}

\newcommand{\eg}{\hat{g}_{m,n}}

\newcommand{\ew}{\hat{W}_{1,n}}
\newcommand{\gm}{\mathcal{G}}

\newcommand{\gmc}{\ssG^c}

\newcommand{\col}[1]{\textcolor{red}{#1}}
\newcommand{\co}[1]{\textcolor{green}{#1}}

\newcommand{\N}{\mathds{N}}

\newcommand{\R}{\mathbb{R}}

\newcommand{\rg}{\relu_G}
\newcommand{\rf}{\relu_D}
\newcommand{\rft}{\relu_D}
\newcommand{\rgt}{\relu_G}

\newcommand{\ep}{\varepsilon}
\renewcommand{\P}{\mathds{P}}

\DeclareMathOperator*{\argmin}{arg\,min}

\newcommandx*{\erw}[2][2= {}]{ \ensuremath{\IE_{#2}\! \left[ #1 \right] }}
\newcommandx*{\eer}[2][2= {}]{ \ensuremath{\hat{\IE}_{#2}\! \left[ #1 \right] }}

\newcommand{\gn}{g^*}

\numberwithin{equation}{section}
\theoremstyle{plain}
\newtheorem{thm}{Theorem}[section]
\newtheorem{lem}[thm]{Lemma}
\newtheorem{cor}[thm]{Corollary}
\newtheorem{prop}[thm]{Proposition}

\theoremstyle{definition}
\newtheorem{defi}[thm]{Definition}
\newtheorem{rem}[thm]{Remark}

\crefname{assum}{assumption}{assumptions}
\crefname{thm}{theorem}{theorems}
\begin{document}

%\iffalse
%\begin{frontmatter}

% "Title of the Paper"
\begin{center}
	\LARGE Statistical analysis of Wasserstein GANs with applications to time series forecasting
\end{center}
%\thankstext{t1}{This is an original survey paper}
%\runtitle{Time series and Wasserstein GANs}

\begin{center}
	\large Moritz Haas and Stefan Richter
\end{center}

%begin{aug}

%\author{\fnms{Moritz} \snm{Haas}\ead[label=e1]{Moritz.Haas@stud.uni-heidelberg.de}}
%\and
%\author{\fnms{Stefan} \snm{Richter}\ead[label=e2]{stefan.richter@iwr.uni-heidelberg.de}}

%\address{Institute of Applied Mathematics\\
%Heidelberg University\\
%\printead{e1,e2}}

%\runauthor{M. Haas and S. Richter}

%\affiliation{Heidelberg University}

%\end{aug}

\begin{abstract}
We provide statistical theory for conditional and unconditional Wasserstein generative adversarial networks (WGANs) in the framework of dependent observations. We prove upper bounds for the excess Bayes risk of the WGAN estimators with respect to a modified Wasserstein-type distance. Furthermore, we formalize and derive statements on the weak convergence of the estimators and use them to develop confidence intervals for new observations. The theory is applied to the special case of  high-dimensional time series forecasting. We analyze the behavior of the estimators in simulations based on synthetic data and investigate a real data example with temperature data. The dependency of the data is quantified with absolutely regular $\beta$-mixing coefficients. 
\end{abstract}

%\begin{keyword}[class=MSC]
%\kwd[Primary ]{62M45}
%\kwd[; secondary ]{62G05}
%\end{keyword}

%\begin{keyword}
%\kwd{Wasserstein GAN}
%\kwd{excess Bayes risk}
%\kwd{convergence rates}
%\kwd{absolutely regular}
%\kwd{high-dimensional}
%\kwd{time series}
%\end{keyword}

% history:
% \received{\smonth{1} \syear{0000}}

%\tableofcontents
%\end{frontmatter}
%\fi

\section{Introduction}

Generative adversarial networks (GANs) are a class of algorithms in machine learning for learning distributions in high-dimensional feature spaces. After the training process, they are able to generate new random fake observations mimicking the observations already seen. In applications, they have shown to provide surprisingly good results in image and speech generation as well as in inpainting tasks.

The training process is designed as follows: Iteratively, two neural networks compete against each other. While the first network (the  \emph{generator}) produces new random observations which imitate the original training samples, the second network (the \emph{critic} or \emph{discriminator}) judges their quality and tries to discriminate between true and generated observations. The assessment is performed with a specific distance of probability distributions. The original GAN was defined with a Kullback-Leibler-type divergence (so called Vanilla GANs, cf. \cite{gan}). In practical applications, GANs using the Wasserstein distance (so called WGANs, cf. \cite{wgan}, \cite{wgangp}) have become popular due to their training stability and the high quality of the generated observations. In contrast to other divergence measures, such as the Kullback-Leibler divergence or the total variation divergence, the Wasserstein distance metrizes weak convergence, which makes it sensible to differences of distributions on lower-dimensional submanifolds and with disjoint supports (cf. \cite{wgan}). This property stabilizes the training procedure of WGANs remarkably.

Let $d\in\IN$ be the dimension of the feature space. WGANs learn a structured probability distribution $\IP^{X}$ from potentially high-dimensional training samples $X_i \in \IR^{d}$, $i \in \{1,...,n\}$. To do so, a latent space $\IR^{d_Z}$ with dimension $d_Z\in\IN$ and latent random variables $Z_1,...,Z_n \in \IR^{d_Z}$ with a given \enquote{base distribution} $\IP^{Z}$ are introduced. Then one tries to minimize the Wasserstein distance of the empirical measure of the training samples,
\[
	\hat\IP_n^{X} := \frac{1}{n}\sum_{i=1}^{n}\delta_{X_i}
\]
(here, $\delta_{X_i}$ denotes the point measure on $X_i$) and the empirical measure of modified latent variables,
\[
	\hat\IP_n^{g(Z)} := \frac{1}{n}\sum_{j=1}^{n}\delta_{g(Z_j)},
\]
with respect to the \emph{generator} $g:\IR^{d_Z} \to \IR^{d_X}$. After the learning process, an estimator $\hat g$ of $g$ can produce new observations $\hat g(Z)$ which approximately follow $\IP^{X}$ by sampling from the latent space $Z \sim \IP^{Z}$.

The approach was generalized to conditional distributions $\IP^{X|Y}$ in \cite{cgans}. Let $d_Y\in\IN$ be the dimension of the conditional feature space. If samples $(X_i,Y_i) \in \IR^{d + d_{Y}}$, $i \in \{1,...,n\}$ are observed, then the conditional WGAN approximately minimizes the Wasserstein distance between the empirical measure
\[
	\hat\IP_n^{X,Y} := \frac{1}{n}\sum_{i=1}^{n}\delta_{X_i,Y_i}
\]
and
\[
	\hat\IP_n^{g(Z,Y),Y} := \frac{1}{n}\sum_{i=1}^{n}\delta_{g(Z_i,Y_i),Y_i},
\]
where here the \emph{conditional generator} is a function $g:\IR^{d_Z + d_Y} \to \IR^{d}$ which also incorporates the values of $Y$ during evaluation. In the same manner as before, approximate observations from $\IP^{X|Y}$ can be obtained after the learning process (i.e. when an estimate $\hat g$ of $g$ is available) by $\hat g(Z,Y)$ with samples $Z \sim \IP^{Z}$ and given observations $Y$. In practice, conditional GANs (cGANs) introduce the information $Y=y$ to the generator in various stages of the architecture. cGANs are very popular for generating images given certain labels such as age, gender or glasses \cite{face_aging} and in image-to-image translation tasks (cf. \cite{imagetoimagecgans,mrireconstr}), e.g. colorizing images or reconstructing higher resolution.

In both situations (unconditional and conditional), WGANs provide an approximation of the law of $\IP^{X}$ or $\IP^{X|Y}$ via $\hat g(Z)$ or $\hat g(Y,Z)$ and offer a powerful tool to obtain new samples even if the training data is high-dimensional. The reason is that under appropriate restrictions on the structure of $g$ and the dimension $d_Z$ of the latent variables, the data $g(Z)$ lies in a low-dimensional submanifold of $\IR^{d}$. 

The purpose of this paper is to provide a theoretical framework for conditional and unconditional WGANs and to prove convergence rates of the excess Bayes risk (with respect to a modified Wasserstein distance) in the context of time series $X_i$, $i=1,...,n$. We formalize in which sense the learned generator function can be used to provide asymptotic confidence sets for $X$. As an application, we will investigate conditional WGANs to provide confidence intervals for observations of high-dimensional time series. The use of WGANs and our corresponding theory is not limited to this example: For instance, one could think of new smoothed  Bootstrap techniques.

%\col{\cite{gan_theory} develop some theory for Vanilla GANs, where not only existence but also uniqueness of the optimizers for convex distribution spaces can be shown, which is not to be expected in a Wasserstein setting. Under smoothness and boundedness assumptions on the parametrization on discriminator, generator and the generated probability distribution, the authors show convergence with rate $\sqrt{n}$ in expectation, but as in \cite{wgan_theory}, the discriminator class is held fixed for any given precision. Therefore the results also lack growth and convergence rates.}
%\col{Wir haben weak convergence theory! Kriegen auch Konfidenzintervalle, und curse of dimension ist weg.}

Recent results from  \cite{gan_theory} and \cite{wgan_theory} already provided theoretical results for the excess Bayes risk of GANs and WGANs in the case of i.i.d. observations $X_i$. They used network classes fixed in $n$ for both discriminators and generators and therefore could not derive convergence rates for the whole excess Bayes risk. Furthermore, the minimized objective could not be used to derive (asymptotic) distributional properties of their corresponding estimators $\hat g$. With our results, we extend the theory of these publications in several ways:
\begin{enumerate}
    \item We derive explicit statistical properties like characterization of weak convergence for the modified Wasserstein distance used in WGANs
    \item We investigate the conditional WGAN, which is an important generalization for standard statistical applications as forecasting.
    \item We allow the generator to be in a H\"older class and explicitly discuss upper bounds for the approximation error of $\hat g$. This yields explicit upper bounds on the whole excess Bayes risk and allows a discussion of the impact of structural assumptions on $g$ and how the curse of dimension can be avoided in practice.
    \item We allow the observations $X_i$, $i=1,...,n$ and $Y_i, i=1,...,n$ to be dependent.
\end{enumerate}
From a technical point of view, we measure dependence with absolutely regular $\beta$-mixing coefficients. We use empirical process theory from \cite{betanorm} and \cite{dl} as well as refined Talagrand's inequalities from \cite{bousquet_bennett} to provide large deviation inequalities of the excess Bayes risk. 

The paper is organized as follows. In Section \ref{sec_definition}, we introduce the Wasserstein metric as well as the conditional and unconditional WGAN estimator based on neural networks.
%\col{In Section \ref{sec_results}, we provide theoretical results for the excess Bayes risk with respect to the Wasserstein distance under structural assumptions on the underlying data generation process and the neural networks used for estimation.} 
\Cref{sec_results} covers the unconditional case. We firstly relate the introduced modified Wasserstein distance (a network based integral probability metric, cf. \cite{integral_prob}) to the 1-Wasserstein distance. Then we provide convergence rates for the excess Bayes risk with respect to this distance under structural assumptions on the underlying data generating process and the neural networks used for estimation. In \Cref{sec_cond_results} we establish equivalent results for the conditional case. In Section \ref{sec_timeseries}, we transfer our results from Section \ref{sec_cond_results} to high-dimensional time series forecasting. In Section \ref{sec:sims}, we provide simulation results of the conditional WGAN algorithm both for simulated data and real-world temperature data. A short conclusion is drawn in Section \ref{sec_conclusion}. All proofs are deferred without further reference to the Appendix.

We now summarize some notation used in this paper.  $(\Omega, \mathcal{A},\P)$ will denote a Borel probability space. For some vector $x\in \IR^d$, let $|x| = (\sum_{j=1}^{d}|x_j|^2)^{1/2}$ denote its Euclidean norm, $|x|_\infty = \max_i |x_i|$ and  $|x|_0= \sum_i \mathbf{1}(x_i\neq 0)$. For measurable functions $f:T \to \IR$, we write $\|f\|_{\infty}:=\sup_{x\in T}|f(x)|$ whenever there is no ambiguity on the domain $T\subset \IR^r$. For $f:T\to \R^{\tilde d}$, we further denote $\|f\|_{\infty} := \max_{j=1,...,\tilde d}\|f_j\|_{\infty}$ and the Lipschitz norm $\|f\|_L := \sup_{x\not=y}\frac{|f(x)-f(y)|}{|x-y|}$ we denote the Lipschitz norm w.r.t. the Euclidean norm $|\cdot|$. Finally, we use the following multi-index calculus: For  differentiable functions $f:T \to \IR$ and $\alpha = (\alpha_1,\dots,\alpha_r)\in\N_0^r$, let $|\alpha|=\sum_{i=1}^r \alpha_i$ and let  $\partial^\alpha f =\partial^{\alpha_1}_1\dots\partial^{\alpha_r}_r f$ denote the $r$-th partial derivative. Finally, for real-valued random variables $W$ and $q > 0$ we write $\|W\|_q := \IE[|W|^{q}]^{1/q}$.

%For measurable functions $f:\IR^d$, we write $P^{\tilde X} f = \IE f(\tilde X)$ and $\hat{P}_n^{\tilde X} f = \frac{1}{n}\smi f(\tilde X_i)$.\\
%Any of the considered probability distributions will be defined on an Polish metric space $\X$ and will lie in the set of Borel probability measures $\mathcal{P}(\X)$.  

\section{The Wasserstein GAN estimator}
\label{sec_definition}

Throughout the paper, we consider $X_i$, $i=1,...,n$ to be a strictly stationary process taking values in $[0,1]^d$, where $d \in\IN$ is an arbitrary dimension. Here, we restrict ourselves to the unit cube $[0,1]^d$ for convenience, our theory could easily be generalized to arbitrary compact Euclidean spaces. We start with an introduction of the typical approximations used in the Wasserstein GAN approach as well as the optimization problem we aim to discuss. Based on this notation, we then give a statistical formulation of the conditional Wasserstein GAN. To keep our assumptions concise, we define the set of functions $f:T \subset\IR^r \to \IR$ with H\"older coefficient $\beta \ge 1$ via
\begin{eqnarray*}
    &&C^{\beta}(T,K)\\
    &:=&\big\{f:T\to \IR \big|\, \sum_{\alpha:\,0 \le |\alpha| < \beta}\|\partial^\alpha f\|_\infty + \sum_{\alpha:\,|\alpha|=\beta-1} \sup_{x\neq y} \frac{|\partial^\alpha f(x)-\partial^\alpha f(y)|}{|x-y|_\infty}\le K \big\},
\end{eqnarray*}
where $K > 0$. %\col{$0 \le |\alpha| < \beta$}

\subsection{Simplification of the WGAN objective and model assumption}\label{sec:wgan_intro} Let $d_Z \in\IN$ and $\IP^{Z}$ a known distribution on $[0,1]^{d_Z}$. Let $\ssG \subset \{g:\IR^{d_Z} \to \IR^{d} \text{ measurable}\}$ be a space of generators. In this paper, we will assume that $\ssG$ consists of smooth functions with a special structure (this is made precise in Definition \ref{generatorclass_wgan}). The objective of a WGAN is to approximate the underlying probability distribution $\IP^{X}$ by minimizing the 1-Wasserstein distance to $\IP^{g(Z)}$ with respect to $g\in \ssG$. By the Kantorovich-Rubinstein duality (cf. \cite{villani}), the 1-Wasserstein distance of two probability distributions $\IP_1, \IP_2$ on $\IR^{d}$ can be written as
\[
    W_1(\IP_1,\IP_2) = \sup_{f:\IR^{d} \to \IR, \|f\|_L\leq 1}\big\{ \int f \dif \IP_1 - \int f\dif \IP_2\big\}.
\]
The original objective of the WGAN is to find a suitable $g \in \ssG$ which minimizes
\begin{equation}
    W_1 (\P^X, \P^{g(Z)}) = \sup_{f:\IR^{d} \to \IR, \|f\|_L\leq 1}\big\{\IE f(X)  - \IE f(g(Z))\big\}.\label{theoretical_objective}
\end{equation}
In practical applications, the set of critics $f:\IR^d \to \IR$ is replaced by a certain set of neural networks $\reluD \subset \{f:\IR^{d} \to \IR\}$. This replacement makes the objective more tractable and also allows the graphical interpretation of competing networks. Similarly, for estimation, the set of possible generators $\ssG$ is replaced by a set of neural networks $\reluG \subset \{f:\IR^{d_Z} \to \IR^{d}\}$.

In the following, we therefore replace the theoretical objective \reff{theoretical_objective} by
\begin{equation}
    W_{1,n}(g) := \sup_{f\in \reluD, \|f\|_L\leq 1}\big\{\IE f(X)  - \IE f(g(Z))\big\}.\label{theoretical_objective_2}
\end{equation}

Note that $W_{1,n}$ may depend on $n$ through the class $\reluD$. Due to the restriction on $f \in \reluD$, one can not expect that $W_{1,n}(g) = W_1(\IP^{X}, \IP^{g(Z)})$. This raises the question which properties $W_{1,n}(g)$ should preserve (and thus, how large and of which form $\reluD$ should be) to make it a meaningful distance of the measures $\IP^{X}$ and $\IP^{g(Z)}$. Our basic aim is to preserve the property that $W_{1,n}$ characterizes weak convergence in the sense that for any sequence $g_n \in \reluG$,
\[
    W_{1,n}(g_n) \to 0\quad\text{ implies }\quad g_n(Z) \overset{d}{\to} X.
\]
The precise conditions on $\reluD$ and its connections to the space $\reluG$ of generators are given in Lemma \ref{lemma_weak_convergence} in Section \ref{sec_results}.

In \cite[Theorem 3]{spike_wasser} it was shown that the 1-Wasserstein distance between two measures in $\IR^d$ can, in general, not be estimated with a better rate than $(n\log(n))^{-1/d}$. Even though $W_{1,n}(g)$ is smaller than $W_1(\IP^{X}, \IP^{g(Z)})$, one needs specific structural assumptions on the underlying distribution and the class of estimators to overcome the curse of dimension. Since we aim to approximate $\IP^{X}$ by $\IP^{g(Z)}$, it is clear that we expect some kind of \enquote{sparsity} of $X$ if $d_Z < d$. If we expect $\IP^{X}$ to lie (approximately) in a $d_g$-dimensional submanifold ($d_g \in \{1,...,d\}$) of $\IR^d$, it seems reasonable to choose $d_Z = d_g$. To allow for a more flexible choice of $d_Z$, we introduce the following function class.

\begin{defi}[Generator function class]\label{generatorclass_wgan}
    Let $\ssG(d_Z,d_g,\beta,K)$ be the set of all measurable functions $g:\IR^{d_Z} \to \IR^{d}$ such that any component only depends on $d_g$ arguments and lies in $C^{\beta}([0,1]^{d_Z},K)$.
\end{defi}

In principle, one can allow for much more complicated structures of the generator functions.  Here we reduce ourselves to the above formulation for simplicity. An example of a more general class  which has auto-encoder structure is introduced for the conditional case (cf. Definition \ref{definition_c_autoencoder}) and could also be chosen here.
%A graphical representation of Assumption \ref{ass_autoencoder} can be found in Figure \ref{fig:crit_gen}. 

\begin{figure}
     \centering
         \includegraphics[width=0.7\textwidth]{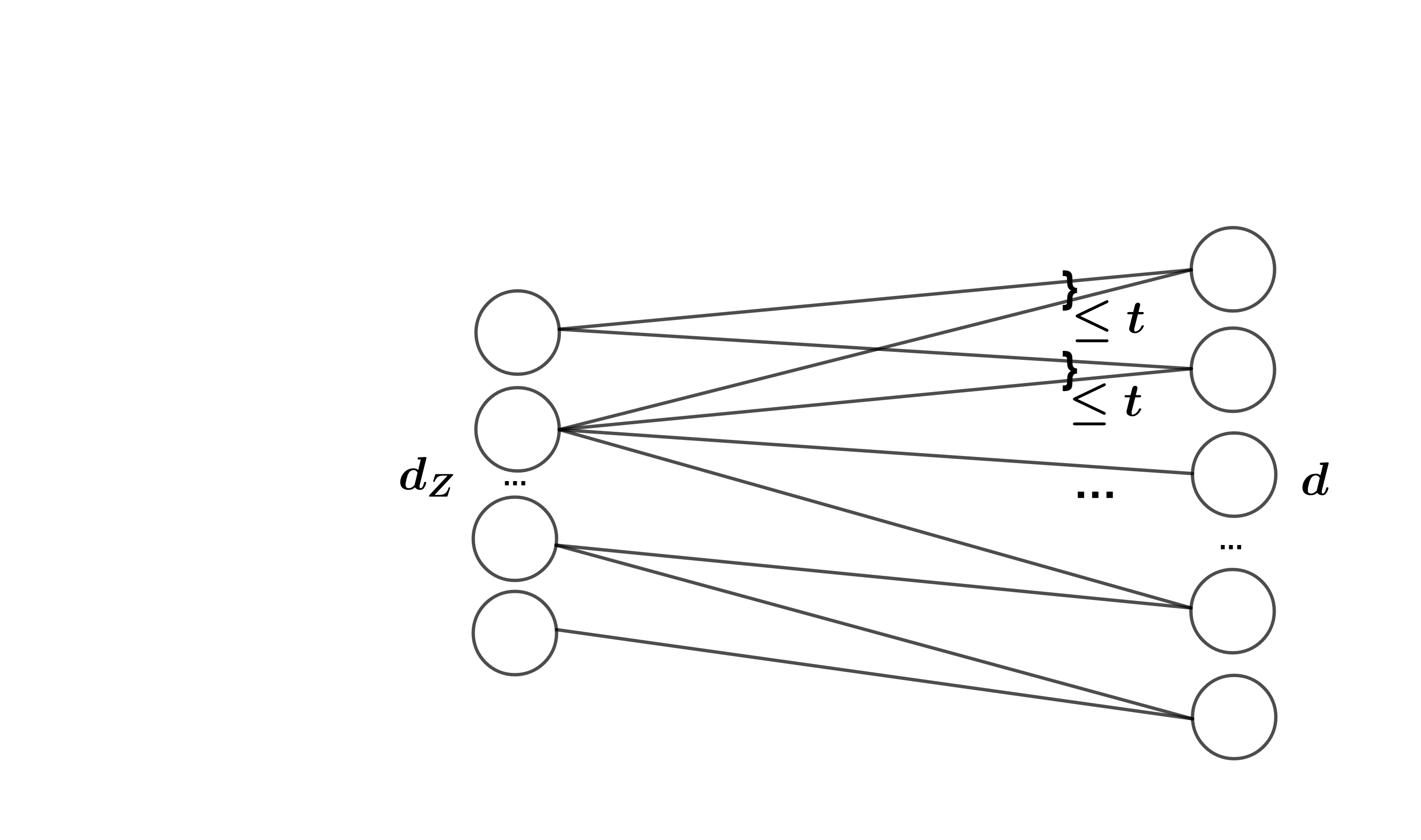}

    \caption{Structure of the generating functions $g$ which are used to model the distribution of $X$ with $g(Z)$.}
    \label{fig:crit_gen}
\end{figure}

\subsection{ReLU neural networks}

We now specify the classes $\reluD$ and $\reluG$ of neural networks in more detail. To do so, we use a theoretical formulation from \cite{sh}. For $x \in\IR$, let  $\sigma(x) = \max\{x,0\}$ denote the rectified linear unit (ReLU) activation function. For $v,x\in R^p$, $p\in\N$, define
\[
    \sigma_{v} (x) = \sigma(x-v),
\]
where $\sigma(\cdot)$ is applied component-wise to the vector $x-v$. Let $L\in\IN$ and $p = (p_0,...,p_{L+1})\in\IN^{L+2}$. A neural network with network architecture $(L,\mathbf{p})$ is a function
\begin{equation}
    h:\R^{p_0}\rightarrow\R^{p_{L+1}}, \quad h(x) = W^{(L)} \sigma_{v^{(L)}} W^{(L-1)} \dots W^{(1)} \sigma_{v^{(1)}} W^{(0)} x,\label{form_network}
\end{equation}
where $W^{(l)} \in \IR^{p_{l} \times p_{l+1}}$, $l=0,...,L$ are the weight matrices and $v^{(l)} \in \IR^{p_l}$, $l=1,...,L$ are the bias vectors associated to the network. Consequently, let
\begin{align*}
    \relu(L,\mathbf{p})=\Big\{ h:\R^{p_0}\rightarrow\R^{p_{L+1}}\;|\; \text{$g$ is of the form \reff{form_network}}\Big\},
\end{align*}
be the class of deep ReLU networks with network architecture $(L,\mathbf{p})$. Training of neural networks typically is done with a stochastic gradient descent method and a random initialization of the weight matrices. It is observed in practice that only few parameters of the resulting networks are \enquote{active} in the sense that they contribute to the final function value. Accordingly, we introduce the set of sparse networks bounded by $F>0$ by
\begin{align*}
    \relu(L,\mathbf{p},s,F):= \Big\{ &h \in \relu(L,\mathbf{p})\;\Big|\; \max_{j=0,\dots,L} \|W_j\|_\infty \vee|v_j|_\infty \leq 1,\\
&\sum_{j=0}^{L} \|W_j\|_0 + |v_j|_0\leq s \text{ and } \| \;|h|_\infty\|_{L^\infty([0,1]^{p_0})}\leq F \Big\}.
\end{align*}
Since $F$ is fixed, we will abbreviate $\relu(L,\mathbf{p},s) = \relu(L,\mathbf{p},s,F)$ in the following.

%To enforce an encoder-decoder structure in practice, we additionally define the set
%\[
 %   \relu(L, L_1, \mathbf{p},s) = \big\{h \in \relu(L,\mathbf{p},s)\;|\; \text{$g$ is of the form \reff{form_network} and $p_{L_1} = \tilde d$}\Big\}.
%\]

\subsection{The unconditional WGAN estimator}\label{sec:relu_classes}

We use the theoretical formulation in \reff{theoretical_objective_2} but replace the expectation $\IE f(X)$ by its empirical counterpart $\frac{1}{n}\sum_{i=1}^{n}f(X_i)$. Furthermore, $\IE f(g(Z))$ is approximated by $\frac{1}{n\Epo}\sum_{j=1}^{n\Epo}f(g(Z_{ij}))$, where $Z_{i,j}$, $j=1,...,\Epo$, $i=1,...,n$ are i.i.d. realizations of $\IP^{Z}$ (independent of $X_i$, $i=1,...,n$) and $\Epo \in\IN$ is some parameter. We then obtain
\begin{equation}
    \hat g_{n} := \argmin_{g\in \mathcal{R}(L_g,\mathbf{p}_g,s_g)}\hat W_{1,n}(g)\label{wgan_unconditional}
\end{equation}
with
\begin{eqnarray*}
    \hat W_{1,n}(g) &:=& \sup_{f\in \mathcal{R}(L_f,\mathbf{p}_f,s_f), \|f\|_L \le 1}\big\{\hat\IP_n^{X}f - \hat\IP_{n\Epo}^{Z}(f \circ g)\big\}\\
    &=& \sup_{f\in \mathcal{R}(L_f,\mathbf{p}_f,s_f), \|f\|_L \le 1}\frac{1}{n}\sum_{i=1}^{n}\big\{f(X_i)  - \frac{1}{\Epo}\sum_{j=1}^{\Epo}f(g(Z_{i,j}))\big\}
\end{eqnarray*}
where $L_g,L_f \in \IN$ are the layer sizes, $\mathbf{p}_g, \mathbf{p}_f$ the corresponding width vectors and $s_g,s_f$ the sparsity parameters.

Note that an optimizer $\hat g_n$ exists (cf. \cite{wgan_theory}), since $\hat W_{1,n}$ is Lipschitz continuous with respect to $g\in\mathcal{R}(L_g,\mathbf{p}_g,s_g)$ and $g\in\mathcal{R}(L_g,\mathbf{p}_g,s_g)$ is Lipschitz continuous with respect to its parameters $W^{(l)}$, $v^{(l)}$, which in turn are defined on a compact set. Similarly there exists an optimal critic network in $\mathcal{R}(L_f,\mathbf{p}_f,s_f)$ for any function $g:[0,1]^{d_Z}\to [0,1]^d$.

The parameter $\Epo$ is motivated by algorithms which are used in practice to find approximations of \reff{wgan_unconditional}, cf. Section \ref{sec:sims}. These algorithms work iteratively. Each iteration which uses all training data is called epoch. The random variables $Z_{i}$, $i=1,...,n$ are not sampled one time at the beginning but new samples are generated in each training epoch. Although in practice the generator only has access to a part of the data $X_i$, $i=1,...,n$ in each epoch, $\Epo$ roughly grows proportional to the number of epochs. Thus one can imitate the knowledge coming from the additional realizations of $\IP^{Z}$ and study its implications.

Appropriate choices for these parameters to guarantee upper bounds for the excess Bayes risk are formulated in Section \ref{sec_results}.

\subsection{The conditional WGAN estimator}\label{sec_cond_set}
Conditional GANs (cGANs), firstly introduced in \cite{cgans}, extend the task of learning to sample from a given distribution $\IP^X$ to learning to sample from conditional distributions $P^{X|Y=y}$, $y\in[0,1]^{d_Y}$, where $Y$ is another random variable in a space $[0,1]^{d_Y}$ encoding some information about $X$. The idea is simply to learn the joint distribution $\IP^{X,Y}$ with the same Wasserstein objective \reff{theoretical_objective} with a generator that has access to $Y$. The formal legitimation is that if we find a function  $g_c^{*}:[0,1]^{d_Z+d_Y} \to \IR^d$ with  $\IP^{X,Y}=\IP^{g_c^*(Z,Y),Y}$, then the independency of $Y,Z$ implies $\IP^{X|Y=y} = \IP^{g_c^*(Z,y)}$.

We introduce a more complex class $\ssG^c(d_Z,d_Y,D,d_g,\beta,K)$ of generators with encoder-decoder structure, cf. Figure \ref{fig:c_crit_gen}, so that generators can depend on all components of the conditional information given by $Y$, even for large dimensions $d_Y$.

\begin{defi}[Encoder-decoder structure]\label{definition_c_autoencoder} Let $\ssG^c(d_Z,d_Y,D,d_g,\beta,K)$ be the set of all measurable functions $g:\IR^{d_Z+d_Y}\to \IR^d$ which have the form \[g=g_{dec}\circ g_{enc,1}\circ g_{enc,0},
\] where
\begin{itemize}
    \item $g_{enc,0}:\IR^{d_Z+d_Y}\to \IR^D$ such that any component only depends on $d_g$ arguments and lies in $C^\beta ([0,1]^{d_g},K)$,
    \item $g_{enc,1}:\IR^{D}\to \IR^{d_g}$ such that any component lies in $C^{\tilde\beta} ([0,1]^{D},K)$, with some $\tilde\beta \geq \frac{D}{d_g} \beta$.
    \item $g_{dec}:\IR^{d_g}\to \IR^{d}$ such that any component lies in $C^\beta ([0,1]^{d_g},K)$.
\end{itemize}
\end{defi}

\begin{figure}
     \centering
        \includegraphics[width=0.7\textwidth]{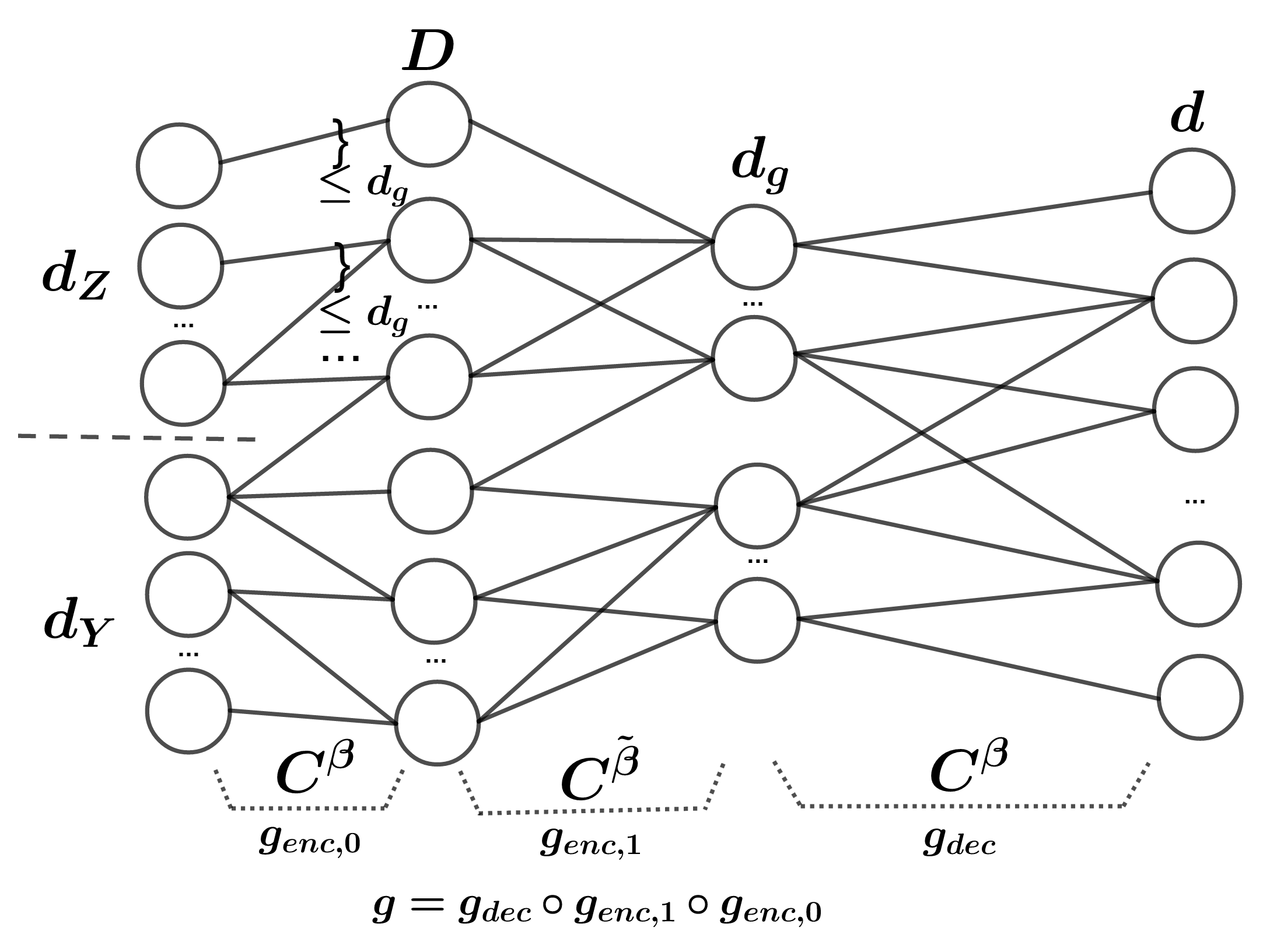}

    \caption{Structure of the generating functions $g$ which are used to model the distribution of $(X,Y)$ with $(g(Z,Y),Y)$.}
    \label{fig:c_crit_gen}
\end{figure}

The original objective now is to find a minimizer of
\[
    W_{1,n}^{c}(g) := \sup_{f \in \mathcal{R}(L_f,\mathbf{p}_f,s_f), \|f\|_L\le 1} \left\{ \IE f(X,Y) -\IE f\big(g(Z,Y),\,Y\big) \right\}.
\]

If $(X_i,Y_i)$, $i=1,\dots,n$ are strictly stationary realizations of $\IP^{(X,Y)}$ and $Z_{i}$, $i=1,...,n$ are i.i.d. realizations of $\IP^{Z}$ independent of $(X_i,Y_i)$, $i=1,...,n$, we define

\begin{equation}
    \hat g_{n}^c := \argmin_{g\in \mathcal{R}(L_g,\mathbf{p}_g,s_g)}\hat W_{1,n}^c(g)\label{wgan_conditional}
\end{equation}
with
\begin{eqnarray*}
    \hat W_{1,n}^c(g) &:=& \sup_{f\in \mathcal{R}(L_f,\mathbf{p}_f,s_f), \|f\|_L \le 1}\frac{1}{n}\sum_{i=1}^{n}\big\{f(X_i,Y_i)  - f(g(Z_{i},Y_i),Y_i)\big\}
\end{eqnarray*}
where $L_g,L_f \in \IN$ are the layer sizes, $\mathbf{p}_g, \mathbf{p}_f$ the corresponding width vectors and $s_g,s_f$ the sparsity parameters. Appropriate choices for these parameters to guarantee upper bounds for the excess Bayes risk are formulated in Section \ref{sec_cond_results}. Note that in contrast to the unconditional WGAN, we do not implement the additional realizations of $\IP^{Z}$ which may occur in practical algorithms. The reason is that for the conditional WGAN, the observations $Y_i$ in the second summand in $\hat W_{1,n}^{c}(g)$ restrict the use of additional knowledge from $Z$ without rather technical assumptions on the structure of $g$.

\section{Theoretical results for the unconditional WGAN}
\label{sec_results}

The first part of this section is devoted to the properties of the modified distances $W_{1,n}(g)$. We show connections between $W_{1,n}(g)$ and distances which do not depend on $n$ and prove that under certain assumptions on the set of networks $\relu_D(L_f,\mathbf{p}_f,s_f)$, $W_{1,n}(g)$ characterizes weak convergence.

In the second part, we provide upper bounds and convergence rates for the excess Bayes risk
\begin{equation}
    R_n(g) := W_{1,n}(g) -  \inf_{g\in \ssG(d_Z,d_g,\beta,K)}W_{1,n}(g)\label{definition_excessbayes}
\end{equation}
for the unconditional WGAN estimator $\hat g_n$ under assumptions on the network structure. In the third part, we summarize the results to provide asymptotic confidence intervals.

\subsection{Properties of the modified Wasserstein distance}

We first investigate the connection of $W_{1,n}(g)$ to
\[
    W_1^{\gamma}(g) := \sup_{f\in C^{\gamma}([0,1]^d,K), \|f\|_L\le 1}\big\{ \IE f(X) - \IE f(g(Z))\big\}.
\]
In opposite to $W_{1,n}$, the quantity $W_1^{\gamma}$ does not depend on $n$ and therefore can be seen as a more \enquote{stable} distance measure for $\IP^{X}$ towards $\IP^{g(Z)}$. Note that
\begin{eqnarray*}
    W_1^{\gamma}(g) &\le& W_{1,n}(g) + 2\sup_{f\in C^{\gamma}([0,1]^d,K)}\inf_{\tilde f \in \mathcal{R}(L_f,\mathbf{p}_f,s_f)}\|f - \tilde f\|_{\infty}.
\end{eqnarray*}

Using approximation results for neural networks from \cite{sh} (cf. Theorem \ref{reluapprox} in the Appendix), one obtains the following result.
%with $m = \lceil\log_2(n)\rceil$, $r = d_g$ and $N = n\phi_n$

\begin{lem}[Lower bound on $W_{1,n}$]\label{lemma_w1beta_approx}
    Let $a_n = n^{-\frac{2\gamma}{2\gamma+d}}$, and suppose that
    \begin{itemize}
        \item $F \ge 1$,
        \item $L_f \ge \log_2(n)\log_2(4d\vee 4\gamma)$,
        \item $\min_{i=1,...,L}p_{f,i} \gtrsim na_n$
        \item $s_f \gtrsim \log(n) na_n$,
    \end{itemize}
    where the constants in the asymptotic expression above depend on $\gamma,d$. Then there exists some constants $C > 0, K \in (0,1)$ only depending on $\gamma,d,F$ such that
    \[
        \sup_{f\in C^{\gamma}([0,1]^d,K), \|f\|_L \le 1}\inf_{\tilde f\in \mathcal{R}_D(L_f,\mathbf{p}_f,s_f)}\|f - \tilde f\|_{\infty} \le C a_n^{1/2}.
    \]
    Especially, for any measurable $g:\IR^{d_Z} \to \IR^d$,
    \begin{equation}
       W_1^{\gamma}(g) \le W_{1,n}(g) + Ca_n^{1/2}.\label{lemma_w1beta_approx_eq1}
    \end{equation}
\end{lem}
The lemma shows that the convergence rate of $W_{1,n}(\hat g_n)$  transfers to $W^{\gamma}_1(\hat g_n)$ as long as $a_n^{1/2} \le W_{1,n}(\hat g_n)$. In fact, this imposes a \emph{lower bound} on the H\"older exponent $\gamma$ of functions considered with $W^{\gamma}_1$.

In the case that $\IP^{X} = \IP^{g^{*}(Z)}$ with some $g^{*} \in \ssG(d_Z,d_g,\beta,K)$, the results for the excess Bayes risk \reff{definition_excessbayes} presented in the following Section \ref{sec_excessbayes} can be used to derive weak convergence. The reasoning is as follows: If $R_n(\hat g_n) \to 0$, then $\IE W_{1,n}(\hat g_n) \to 0$. Then the following lemma can be used.

\begin{lem}[Characterization of weak convergence]\label{lemma_weak_convergence} Suppose that $\IP^{X} = \IP^{g^{*}(Z)}$ for some $g^{*} \in \ssG(d_Z,d_g,\beta,K)$ and let the assumptions of Lemma \ref{lemma_w1beta_approx} hold with some $\gamma \ge 1$. Let $(\hat g_n)_{n\in\IN}$ be a sequence of random variables with $\IE W_{1,n}(\hat g_n) \to 0$. Then
\[
    \hat g_n(Z) \overset{d}{\to} g^{*}(Z) = X.
\]
\end{lem}

The lemma basically follows from Lemma \ref{lemma_w1beta_approx} and the fact that $C^{\gamma}([0,1]^d)$ forms a convergence-determining class.

\begin{rem}\label{remark_weakconvergence}
    Lemma \ref{lemma_weak_convergence} implies weak convergence of $\IP^{\hat g_n(Z)}$ towards $\IP^{X}$, but it does not give any information about the speed of convergence. Nevertheless it seems reasonable that the speed depends on the upper bound in \reff{lemma_w1beta_approx_eq1} which is given by the two summands $W_{1,n}(g)$ and $a_n^{1/2}$. Therefore, one should choose $\gamma \ge 1$ large enough such that $a_n^{1/2} \lesssim W_{1,n}(g)$. This is done in Remark \ref{remark_critic_generator} below. On the other hand, for larger $\gamma$,  $W_1^{\gamma}(g)$ gives less information about the distance between $\IP^{g(Z)}$ and $\IP^{X}$. It is an open question how an optimal balance of $\gamma$ should be chosen. A more detailed analysis of the approximation quality of the set of critic networks $\mathcal{R}_D(L_f,\mathbf{p}_f,s_f)$ could yield more insight. However, this would need sharp upper bounds on the Lipschitz constants of $\mathcal{R}(L_f,\mathbf{p}_f,s_f)$ and is a pure approximation problem, which is out of the scope of this paper.
\end{rem}

\subsection{Excess Bayes risk}
\label{sec_excessbayes}

To state the theoretical results on the excess Bayes risk $R_n(\hat g_n)$ in \reff{definition_excessbayes}, we have to quantify the dependence structure of $X_i$, $i=1,\dots, n$ and $Y_i$, $i=1,...,n$. Basically, observations obtained at time steps which are far away from each other have to be \enquote{asymptotically independent}. There exists a large variety of weak and strong mixing conditions. We refer to \cite{bradley_mixing} for a detailed summary of conditions and basic properties. Here, we use absolutely regular $\beta$-mixing due to the well-established empirical process theory (cf. \cite{betanorm} and \cite{dl}).

The $\beta$-mixing coefficient between two $\sigma$-algebras $\mathcal{U}, \mathcal{V}\subseteq\mathcal{A}$ is defined by
\[
    \beta(\mathcal{U},\mathcal{V}):=\frac{1}{2} \sup \sum_{(i,j)\in I\times J} \big| \P(U_i\cap V_j)-\P(U_i) \P(V_j)\big|,
\]
where the supremum is taken over all finite partitions $(A_i)$ and $(B_j)$ $\mathcal{U}$- and $\mathcal{V}$-measurable respectively. For a time series $X_i$, $i=1,...,n$, one defines
\[
    \beta_{ X}(0)=1,\quad
    \beta_{ X}(n):= \beta\big(\sigma({X}_i;\, i\le 0), \sigma({X}_i; \, i\geq n)\big), \quad n\in \N.
\]
Prominent examples of absolutely regular sequences are GARCH and ARMA as well as linear processes (cf. \cite{mixing_arch,bradley_mixing,doukhan2012mixing}).

We now present the theoretical result for the excess Bayes risk for the unconditional WGAN estimator $\hat g_n$.

\begin{thm}\label{theorem_wgan_excessbayes}
       Let $\phi_n = (n\Epo)^{-\frac{2\beta}{2\beta + d_g}}$. Suppose that $F \ge K \vee 1$, and
    \begin{itemize}
        \item[(i)] $\log_2(n\Epo) \log_2(4d_g \vee 4\beta) \le L_g \lesssim \log(n\Epo)$,
        \item[(ii)] $min_{i=1,...,L_g}p_{g,i} \gtrsim n\Epo\phi_n$,
        \item[(iii)] $s_g \asymp n\Epo\phi_n \log(n\Epo)$
        \item[(iv)] $L_f \le L_g$, $s_f \le s_g$.
    \end{itemize}
    Suppose that there exist constants $\kappa > 1, \alpha > 1$ such that for all $k\in\IN$, $\beta_X(k) \le \kappa\cdot k^{-\alpha}$.
    Then
    \begin{equation}
        \IE R_n(\hat g_n) \lesssim \Big(\frac{s_f L_f \log(s_f L_f)}{n}\Big)^{1/2} + \phi_n^{1/2}\log(n\Epo)^{3/2},\label{theorem_wgan_excessbayes_res1}
    \end{equation}
    and with probability at least $1-4n^{-1}-2(\frac{\log(n)}{n})^{\frac{\alpha-1}{2}}$,
    \[
        R_n(\hat g_n) \lesssim \Big(\frac{s_f L_f \log(s_f L_f)}{n}\Big)^{1/2} + \phi_n^{1/2}\log(n\Epo)^{3/2} + \Big(\frac{\log(n)}{n}\Big)^{1/2},
    \]
    where the bounding constants may depend on characteristics of $X_1$, $\kappa,\alpha$ and $d,d_Z,d_g,\beta,K,F$.
\end{thm}

\textbf{Remark on dependency}. Let us discuss this result in more detail. First note that we basically only need that $\beta_X(k)$ is summable. The specific polynomial rate of the decay only enters the large deviation result, and is negligible if $\alpha \ge 3$. The statistical bounds are obtained by using empirical process theory for absolutely regular $\beta$-mixing sequences and refined versions of Talagrand's inequality for independent variables from \cite{klein2005}.

\textbf{Remark on critic networks}. The rate in \reff{theorem_wgan_excessbayes_res1} decomposes into two terms, where the first term is determined by properties of the critic networks and the second term stems from generator networks. Regarding the first term, the only condition on the critic functions is given in (iv) which asks that the critic networks allow for less non-zero parameters and have less layers than the generator networks. The lack of conditions is clear since there is no a priori  approximation task the critic networks have to fulfill in $W_{1,n}$. However, there is some interest in allowing for a large critic function class $\mathcal{R}(L_f,\mathbf{p}_f,s_f)$ due to the results from Lemma \ref{lemma_w1beta_approx}. This is discussed in more detail in Remark \ref{remark_critic_generator}.

\textbf{Remark on conditions}. Assumptions (i)-(iii) stated in Theorem \ref{theorem_wgan_excessbayes} are conditions on the network structure of the generator networks which are allowed (and needed) to grow with $n$. Basically, the lower bounds on $L_g,\mathbf{p}_g,s_g$ are used to bound the approximation error of finding an element $\tilde g \in \mathcal{R}_G(L_g,\mathbf{p}_g,s_g)$ which approximates $g \in \ssG(d_Z,d_g,\beta,K)$ well,  while the upper bounds control the estimation error.

(i),(iii) ask the generator networks to have approximately $\log_2(n\Epo)$ layers and allow for approximately
\[
    n\Epo \phi_n \log(n\Epo) = (n\Epo)^{\frac{d_g}{2\beta+d_g}}\log(n\Epo)
\]
non-zero parameters (that is, entries in weight matrices and bias vectors). (ii) asks the layers to have a certain minimal width. Letting the minimal width of \emph{all} layers grow polynomially in $n$ seems rather unusual from a practical point of view. This is only due to the approximation technique adopted from \cite{sh} and can be improved. 

\textbf{Remark on convergence rate - dimensionality}. In \cite[Theorem 1]{minimaxwhoelder} (cf. also \cite{nilesweed2019minimax}) it was shown that $\beta$-H\"older smooth densities in the space $\IR^{d}$ can be estimated with a rate not faster than $n^{-\frac{\beta+1}{2\beta+d}}$ with respect to the Wasserstein distance. Note the additional $\beta+1$ in the nominator instead of $\beta$ as it is the case, for instance, in standard nonparametric density estimation. Due to the additional structural assumptions, out method yields (up to a log factor) a convergence rate $\phi_n = n^{-\frac{\beta}{2\beta+d_g}}$ with respect to the modified Wasserstein-distance $W_{1,n}(\hat g_n)$. It does not depend on the underlying dimensionality $d_Z$ of the generation space nor the dimension $d$ of the observation space but only on the reduced dimension $d_g \le d_Z$. Even if $d_Z$ is chosen large (as it may occur in practice), the generator network estimator $\hat g_n$ can adapt to the unknown number $d_g$ of relevant arguments without suffering from a curse of dimension.

%The additional $\beta+1$ in the nominator (instead of $\beta$) as it is more usual for nonparametric estimation is due to the fact that imposing smoothness assumptions on the density $\IP^{g(Z)}$ is stronger than imposing smoothness assumptions on the generator $g$ itself.  The rate $\phi_n = n^{-\frac{\beta}{2\beta+d_g}}$ therefore indicates that 

\textbf{Remark on convergence rate - generator size}. In practice, along with each sampled batch of data one batch of generated data $Z_{i1}$, $i=1,...,n$ is produced, so that during the first epoch of training it holds that $\Epo = 1$.  In subsequent epochs, the data set of fixed size $n$ is reused, while the number of generated samples keeps growing. If we assume that the generator networks approximate the empirical optimizers at each step, the variable $\Epo$ introduced in the estimator $\hat g_n$ can be roughly seen as the number of training epochs and indicates that more and more realizations of $\IP^{Z}$ are available to train $\hat g_n$. In principle, $\Epo$ can be chosen arbitrarily large, therefore one can use arbitrarily large generator architectures as long as one generates enough samples during training. Then the performance saturates due to the limited data samples $n$ and the corresponding discriminator architecture (cf. \reff{theorem_wgan_excessbayes_res1}), but not due to the generator capacity. However, note that one cannot directly take $\Epo$ as the number of epochs. The reason is that in each epoch, the generator only sees a part of the data $X_i$, $i=1,\dots,n$ (see \Cref{table:WGANGD}).

\begin{rem}[Selection of critic and generator]\label{remark_critic_generator}
    If there exists $g^{*} \in \ssG(d_Z,d_g,\beta,K)$ with $\IP^{g^{*}(Z)} = \IP^{X}$, then the results of Theorem \ref{theorem_wgan_excessbayes},  \reff{theorem_wgan_excessbayes_res1} and Lemma \ref{lemma_w1beta_approx}, \reff{lemma_w1beta_approx_eq1} can be combined. In this case, one could ask for a suitable choice of $\gamma \ge 1$ such that the rates coincide, that is, $a_n = \phi_n$. This then also leads to more precise conditions on the discriminator network through Lemma \ref{lemma_w1beta_approx}. For simplicity, choose $\Epo = 1$. We see that equality is obtained with
    \[
        \frac{\beta}{2\beta+d_g} = \frac{\gamma}{2\gamma+d},
    \]
    which is fulfilled for $\gamma = \beta\frac{d}{d_g}$, and leads to
    \[
        \IE W^{\gamma}_1(\hat g_n) \lesssim \phi_n^{1/2}\log(n)^{3/2}.
    \]
\end{rem}

\subsection{Asymptotic confidence intervals}\label{sec_uncon_conf_intervals}

Based on the weak convergence, one can provide asymptotic confidence sets for $X$ to a given level $\alpha$. For simplicity, suppose that $X$ is one-dimensional. For $N\in\IN$, let $Z_j^{*}, j=1,...,N$ be i.i.d. samples of $\IP^{Z}$, independent of $X_i,Z_{ij}$ used to calculate the WGAN estimator $\hat g_n$ from \reff{wgan_unconditional}. Define the empirical distribution function
\[
    \hat F_{N,n}(x) := \frac{1}{N}\sum_{j=1}^{N}\Ii_{\{\hat g_n(Z_j^{*}) \le x\}}
\]
and let $F_X$ denote the distribution function of $X$. Then the following result holds.

\begin{lem}\label{lemma_weak_convergence_distribution}
    Suppose that $\IP^{X} = \IP^{g^{*}(Z)}$ for some $g^{*} \in \ssG(d_Z,d_g,\beta,K)$ and that $\IP^{X}$ is continuous. Let the assumptions of Lemma \ref{lemma_w1beta_approx} with some $\gamma \ge 1$ and Theorem \ref{theorem_wgan_excessbayes} hold. Then for any $\rho > 0$,
    \[
        \limsup_{n\to\infty}\limsup_{N\to\infty}\IP(|\hat F_{N,n}(X) - F_X(X)| \ge \rho) = 0.
    \]
\end{lem}
By the probability integral transform, this shows that $\hat F_{N,n}(X)$ converges in probability to a uniform distribution on $[0,1]$. For fixed $\alpha \in (0,1)$, this justifies that the interval
\begin{equation}
    I_{n,N} := \Big\{x\in \IR: \hat F_{N,n}(x) \in \left(\frac{\alpha}{2},1-\frac{\alpha}{2}\right]\Big\} \label{eq_uncon_conf_int} 
\end{equation}
   
which is built from the empirical $\frac{\alpha}{2}$ and $(1-\frac{\alpha}{2})$ quantile curves of $\hat g_n(Z_j^{*})$, $j=1,...,N$ is an asymptotic $(1-\alpha)$-confidence set for $X$ since
\[
    \IP(X \in I_{n,N}) = \IP\left(\hat F_{N,n}(X) \in \left(\frac{\alpha}{2},1-\frac{\alpha}{2}\right]\right) \approx \IP\left(F_X(X) \in \left(\frac{\alpha}{2},1-\frac{\alpha}{2}\right]\right) = 1-\alpha.
\]

\section{Theoretical results for the conditional WGAN}\label{sec_cond_results}

\subsection{Results for the modified conditional Wasserstein distance}

We now provide similar results as given in Lemma \ref{lemma_w1beta_approx} and Lemma \ref{lemma_weak_convergence} for the conditional WGAN formulation.

In analogy, firstly define,
\[
    W_1^{c,\gamma}(g) := \sup_{f\in C^{\gamma}([0,1]^{d+d_Y},K), \|f\|_L\le 1}\big\{ \IE f(X,Y) - \IE f(g(Z,Y),Y)\big\}.
\]

Since we use the same approximation result \cite{sh}, the connection of $W^c_{1,n}(g)$ to $W_1^{c,\gamma}(g)$ is essentially the same as in the unconditional case and we omit a proof.
\begin{lem}\label{lemma_w1beta_approx_conditional}
    Let $a_n = n^{-\frac{2\gamma}{2\gamma+d+d_Y}}$, and suppose that
    \begin{itemize}
        \item $F \ge 1$,
        \item $L_f \ge \log_2(n)\log_2(4(d+d_Y)\vee 4\gamma)$,
        \item $\min_{i=1,...,L}p_{f,i} \gtrsim na_n$
        \item $s_f \gtrsim \log(n) na_n$,
    \end{itemize}
    where the constants in the asymptotic expression above depend on $\gamma,d,d_Y$. Then there exists some constants $C > 0, K \in (0,1)$ only depending on $\gamma,d,d_Y,F$ such that
    \[
        \sup_{f\in C^{\gamma}([0,1]^{d+d_Y},K), \|f\|_L \le 1}\inf_{\tilde f\in \mathcal{R}_D(L_f,\mathbf{p}_f,s_f)}\|f - \tilde f\|_{\infty} \le C a_n^{1/2}.
    \]
    Especially, for any measurable $g:\IR^{d_Z} \to \IR^d$,
    \begin{equation}
       W_1^{c,\gamma}(g) \le W^c_{1,n}(g) + Ca_n^{1/2}.\label{lemma_w1beta_approx_eq1_cond}
    \end{equation}
\end{lem}

%\subsubsection{Characterization of weak convergence}
In the case that $\IP^{X,Y} = \IP^{g^{*}(Z,Y),Y}$ with some $g^{*} \in \ssG^c(d_Z,d_Y,D,d_g,\beta,K)$, the results for the excess Bayes risk \reff{definition_excessbayes_conditional} presented in the following Section \ref{sec:cond_set} can be used to derive weak convergence with the help of the following lemma.

\begin{lem}\label{lemma_weak_convergence_conditional} Suppose that $\IP^{X,Y} = \IP^{g^{*}(Z,Y),Y}$ for some $g^{*} \in \ssG^c(d_Z,d_Y,D,d_g,\beta,K)$ and let the assumptions of Lemma \ref{lemma_w1beta_approx} hold with some $\gamma \ge 1$. Let $(\hat g^c_n)_{n\in\IN}$ be a sequence of random variables with $\IE W^c_{1,n}(\hat g_n^c) \to 0$. Then
\begin{equation}
    \hat g^c_n(Z,Y) \overset{d}{\to} g^{*}(Z,Y) = X.\label{convergence_weak_joint}
\end{equation}
\end{lem}

The lemma basically follows from \Cref{lemma_w1beta_approx_conditional} and the fact that $C^{\gamma}([0,1]^{d+d_Y})$ forms a convergence-determining class. Remark \ref{remark_weakconvergence} applies here as well. Moreover, from \reff{convergence_weak_joint} one directly obtains the convergence of $\IP^{\hat g^c_n(Z,y)}$ towards the conditional distribution $\IP^{X|Y = y}$.

\subsection{Excess Bayes risk}\label{sec:cond_set}

We now provide a result for the excess Bayes risk of the conditional WGAN,
\begin{equation}
    R_n^c(g) := W_{1,n}^c(g) -  \inf_{g\in \ssG(d_Z,d_Y,D,d_g,\beta,K)}W_{1,n}^c(g).\label{definition_excessbayes_conditional}
\end{equation}

\begin{thm}\label{theorem_wgan_conditional_excessbayes}
Let $\phi_n = n^{-\frac{2\beta}{2\beta + d_g}}$ and $\tilde \beta \ge \frac{D}{d_g}\beta$. Suppose that $F \ge K \vee 1$, and 
    \begin{itemize}
        \item[(i)] $\log_2(n) \big(2\log_2(4d_g \vee 4\beta)+\log_2(4 D \vee 4 \tilde \beta)\big) \le L_g \lesssim \log(n)$,
        \item[(ii)] $min_{i=1,...,L_g}p_{g,i} \gtrsim n\phi_n$,
        \item[(iii)] $s_g \asymp n\phi_n \log(n)$
        \item[(iv)] $L_f \le L_g$, $s_f \le s_g$.
    \end{itemize}
    Suppose that there exist constants $\kappa > 1, \alpha > 1$ such that for all $k\in\IN$, $\beta_{X,Y}(k) \le \kappa\cdot k^{-\alpha}$.
    Then
    \begin{equation}
        \IE R^c_n(\hat g^c_n) \lesssim \Big(\frac{s_f L_f \log(s_f L_f)}{n}\Big)^{1/2} + \phi_n^{1/2}\log(n)^{3/2},\label{theorem_wgan_conditional_excessbayes_res1}
    \end{equation}
    and with probability at least $1-4n^{-1}-2(\frac{\log(n)}{n})^{\frac{\alpha-1}{2}}$,
    \[
        R^c_n(\hat g^c_n) \lesssim \Big(\frac{s_f L_f \log(s_f L_f)}{n}\Big)^{1/2} + \phi_n^{1/2}\log(n)^{3/2} + \Big(\frac{\log(n)}{n}\Big)^{1/2},
    \]
    where the bounding constants may depend on characteristics of $(X_1, Y_1)$ and $\kappa,\alpha,d,d_Z,d_Y, D, d_g,\beta,\tilde \beta,K,F$.
\end{thm}

All remarks for \Cref{theorem_wgan_excessbayes} apply here as well.

\subsection{Asymptotic confidence intervals}
\label{sec_asympt_conf_cwgan}

For simplicity, suppose that $X$ is one-dimensional. For $N\in\IN$, let $Z_j^{*}, j=1,...,N$ be i.i.d. samples of $\IP^{Z}$, independent of $X_i,Y_i,Z_{i}$ used to calculate the WGAN estimator $\hat g_n^c$ from \reff{wgan_conditional}. Define the empirical distribution function
\[
    \hat F_{N,n}^c(x|y) := \frac{1}{N}\sum_{j=1}^{N}\Ii_{\{\hat g_n^c(Z_j^{*},y) \le x\}}
\]
and let $F_{X}(x|y) = \IP(X \le x|Y = y)$ denote the distribution function of $X$ conditional on $Y = y$. The proof of the following result is similar to Lemma \ref{lemma_weak_convergence_conditional} and therefore omitted.

\begin{lem}\label{lemma_weak_convergence_distribution_conditional}
    Suppose that $\IP^{(X,Y)} = \IP^{(g^{*}(Z,Y),Y)}$ for some $g^{*} \in \ssG(d_Z,d_Y,d_g,D,\beta,K)$ and that $F_X(x|y)$ is continuous for $\IP^{Y}$-a.e. $y$. Let the assumptions of Lemma \ref{lemma_w1beta_approx_conditional} with some $\gamma \ge 1$ and Theorem \ref{theorem_wgan_conditional_excessbayes} hold. Then for any $\rho > 0$ and for $\IP^{Y}$-a.e. $y$,
    \[
        \limsup_{n\to\infty}\limsup_{N\to\infty}\IP\big(\big|\hat F_{N,n}(X|y) - F_X(X|y)\big| \ge \rho\big|Y = y\big) = 0.
    \]
\end{lem}
As before in the case of the unconditional WGAN, we can now construct an asymptotic $(1-\alpha)$ confidence set for $X$ conditional on $Y = y$. For fixed $\alpha \in (0,1)$, let
\begin{equation}
    I_{n,N}(y) := \Big\{x\in \IR: \hat F_{N,n}(x|y) \in \Big(\frac{\alpha}{2},1-\frac{\alpha}{2}\Big]\Big\}.\label{conf_cwgan}
\end{equation}
Then for large $n,N$, one has
\begin{eqnarray*}
    \IP(X \in I_{n,N}(y)|Y = y) &=& \IP(\hat F_{N,n}(X|y) \in \left(\frac{\alpha}{2},1-\frac{\alpha}{2}\right]|Y = y)\\
    &\approx& \IP(F_{X}(X|y) \in \left(\frac{\alpha}{2},1-\frac{\alpha}{2}\right]|Y = y) = 1-\alpha.
\end{eqnarray*}

\section{High-dimensional time series forecasting}
\label{sec_timeseries}

Earlier practical approaches of \cite{tcgan,tcwgan} have shown that conditional WGANs can be used to determine distributional forecasts of time series. In this section we use our results to provide asymptotic confidence intervals.

Suppose that we have given a time series $A_i \in \IR^{p}$, $i=-r+1,...,n$ with \emph{continuous} distribution which is absolutely regular $\beta$-mixing with coefficients $\beta_A(k)$, $k \ge 0$. We are interested in forecasting a statistic
\[
    T(A_{i}),\quad\quad T:\IR^d \to \IR \text{ continuous},
\]
conditional on the finite past
\[
    \IA_{i-1} := (A_{i-1},...,A_{i-r}) \in \IR^{pr},
\]
where $r \in\IN$ denotes the number of lags considered. Let us furthermore assume that there exists some $\beta \ge 1, K > 0$, $d_g \in\IN$ and $g^{*c} \in \ssG^c(d_Z,pr,d_g,\beta,K)$ (cf. Definition \ref{definition_c_autoencoder}) such that
\[
    \IP^{T(A_r), \IA_{r-1}} = \IP^{g^{*}(Z,\IA_{r-1})},
\]
that is, the distribution of $T(A_r)$ is obtained from $A_{r-1},...,A_1$ and some random noise $Z$. Let $\hat g_n^c$ denote the conditional WGAN estimator from \reff{wgan_conditional}, that is,
\[
    \hat g_n^c = \argmin_{g\in \mathcal{R}(L_g,\mathbf{p}_g,s_g)}\hat W_{1,n}^c(g)
\]
with
\[
    \hat W_{1,n}^c(g) := \sup_{f\in \mathcal{R}(L_f,\mathbf{p}_f,s_f), \|f\|_L \le 1}\frac{1}{n}\sum_{i=1}^{n}\big\{f(T(A_i),\IA_{i-1}) - f(g(Z_i,\IA_{i-1}),\IA_{i-1})\big\}.
\]
With the above definitions, $(T(A_i),\IA_{i-1})$ is absolutely regular $\beta$-mixing with coefficients $\beta(k) = \beta_A((k-r)\vee 0)$. Then Theorem \ref{theorem_wgan_conditional_excessbayes} implies:

\begin{cor}
    Under the conditions (i)-(iv) of Theorem \ref{theorem_wgan_conditional_excessbayes},
    \[
        \IE W_{1,n}^c(\hat g_n^c) \lesssim \Big(\frac{s_f L_f \log(s_f L_f)}{n}\Big)^{1/2} + n^{-\frac{\beta}{2\beta+d_g}}\log(n)^{3/2}.
    \]
\end{cor}

We obtain confidence intervals for $T(A_i)$ given $\IA_{i-1} = \ia$ as follows based on the results for conditional WGANs from Section \ref{sec_asympt_conf_cwgan}. For $N\in\IN$, let $Z_1^{*},...,Z_N^{*}$ denote i.i.d. realizations of $\IP^{Z}$. Define the empirical distribution function given $\IA_{i-1} = \ia$,
\[
    \hat F_{N,n}(t|\ia) := \frac{1}{N}\sum_{j=1}^{N}\Ii_{\{\hat g_n^c(Z_j^{*},\ia) \le t\}}.
\]
Then for $\alpha \in (0,1)$, the interval
\[
    I_{N,n}(\ia) := \Big\{t \in \IR: \hat F_{N,n}(t|\ia) \in \left(\frac{\alpha}{2},1-\frac{\alpha}{2}\right]\Big\}
\]
is an asymptotic $(1-\alpha)$ confidence interval for $T(A_i)$ given $\IA_{i-1} = \ia$.

\section{Simulation studies}\label{sec:sims}
In this section we study the behaviour of an approximation of the optimal estimators $\hat g_n$ from \reff{wgan_unconditional} and $\hat g_n^{c}$ from \reff{wgan_conditional} obtained by gradient descent methods. In all our experiments we use the WGAN-GP \cite{wgangp} algorithm with the adapted default values given in \Cref{table:WGANGD}, if not stated otherwise.

\begin{table}[h!]
\centering
\begin{tabular}{ |c l| }
\hline
\multicolumn{2}{|c|}{\textbf{WGAN-GP.}} \\
\hline
\multicolumn{2}{|c|}{\textbf{Require:} $\beta_1$, $\beta_2$, learning rate $\alpha$, penalty weight $\lambda$, batch size $m$,}\\
\multicolumn{2}{|c|}{ number of critic iterations per generator iteration $n_{\text{critic}}$.}\\
\hline
0: & Initialize critic parameters $\theta_{\text{critic}}$ and generator parameters $\theta_{\text{gen}}$.\\
 1: & \textbf{while} $\theta_{\text{gen}}$ has not converged: \\ 
 2: & \quad \textbf{for} $t = 0,\dots, n_{\text{critic}}$:\\ 
 3: & \qquad Sample a batch $\{ X^{(i)}\}_{i=1}^m\sim \P^X$ from the real data.\\
 4: & \qquad Sample i.i.d. batches $\{ Z^{(i)}\}_{i=1}^m\sim\P^Z$, \; $\{ U^{(i)}\}_{i=1}^m\sim U[0,1]$.\\
 5: & \qquad Compute $\tilde X^{(i)} = U^{(i)} X^{(i)} + (1-U^{(i)}) g_{\theta_{\text{gen}}}(Z^{(i)})$.\\
 6: & \qquad $G_{\text{c}} \leftarrow \nabla_{\theta_{\text{critic}}} \big( \frac{1}{m}\sum_{i=1}^m f_{\theta_{\text{critic}}}(X^{(i)})-\frac{1}{m}\sum_{i=1}^m f_{\theta_{\text{critic}}}(g_{\theta_{\text{gen}}}(Z^{(i)})) \big).$ \\
 7: & \qquad $Pen_{\text{c}} \leftarrow \lambda \cdot \frac{1}{m}\sum_{i=1}^m \big( ||\nabla_{\tilde X_i} f_{\theta_{\text{critic}}}||_2 - 1 \big)^2$.\\
 8: & \qquad $\theta_{\text{critic}} \leftarrow \theta_{\text{critic}} + \alpha \cdot \text{ADAM}(G_c+Pen_{\text{c}},\,\theta_{\text{critic}},\,\beta_1,\,\beta_2).$\\
 9: & \quad \textbf{end for}\\
 10: & \quad Sample an i.i.d. batch $\{ Z^{(i)}\}_{i=1}^m\sim\P^Z$.\\
 11: & \quad $G_{\text{gen}} \leftarrow -\nabla_{\theta_{\text{gen}}} \frac{1}{m}\sum_{i=1}^m f_{\theta_{\text{critic}}}(g_{\theta_{\text{gen}}}(Z^{(i)})) \big).$ \\
 12: & \quad $\theta_{\text{gen}} \leftarrow \theta_{\text{gen}} - \alpha \cdot \text{ADAM}(G_{\text{gen}},\,\theta_{\text{gen}},\,\beta_1,\,\beta_2).$\\
 13: & \textbf{end while}\\
 \hline
\end{tabular}
\caption{The gradient descent algorithm proposed in \cite{wgangp}, with adapted default values $\alpha=0.0001$, $\lambda=0.1$, $m=64$, $n_{\text{critic}} = 5$, $\beta_1=0.5$, $\beta_2=0.9$ and $\P^Z=U[0,1]$. The critic tries to maximize the empirical Wasserstein distance, while the generator has the contrary objective. The penalty term in line 7 softly enforces the Lipschitz constraint on the critic. For the first 25 and every $100^{th}$ generator iterations, we train the critic for 100 iterations for each generator iteration to ensure critic convergence and meaningful gradients. We additionally apply $0.01$ $L_2$-weight decay to both networks to softly enforce boundedness of the network parameters.}
\label{table:WGANGD}
\end{table}

We now give some comments on the WGAN-GP algorithm. In opposite to the original WGAN algorithm from \cite{wgan} (cf. \cite[Theorem 1]{anil2018sorting}), which uses crude weight clipping to guarantee a bounded Lipschitz constant of the critic networks, WGAN-GP realizes the Lipschitz constraint in the definition of $\hat W_{1,n}$ via a penalty term. Furthermore, the critic is learned from $\hat W_{1,n}$ using a gradient descent method. Although the critic network may not obey $\|f\|_L\le 1$ by using a penalty term, the equation
\[
L \cdot \sup_{\|f\|_L\le 1} \IE f(X)-\IE f(g(Z)) = \sup_{\|f\|_L\le L} \IE f(X)-\IE f(g(Z)),
\]
shows that it is enough to bound the Lipschitz constant of the critic by some (unknown) constant.  WGAN-GP therefore recovers the distributional stability of the $W_1$-distance, induced by metrizing weak convergence (cf. also \cite{arjovsky2017principled}). However, it should be noted that the new latent variables $Z_{ij}$ generated in each training epoch may slightly change the bound on the Lipschitz constant introduced by the penalty term.

For Vanilla GANs generator and discriminator training have to be carefully balanced, because a discriminator that classifies too well does not yield informative gradients (the so-called saturation phenomenon). For WGANs, in contrast, better critics yield better gradients. Hence one only has to train the critic \enquote{long enough}, which is a huge practical advantage. We can confirm stable training behaviour in all our experiments.

%Now, if the critic network approximates the empirically optimal critic well enough, apart from the Lipschitz constant, then the generator network learns to approximate the empirical optimizer $\eg$ via gradient descent.

For the conditional setting, we simply replace all $f_{\theta}(X^{(i)})$ by $f_{\theta}(X^{(i)}, Y^{(i)})$ and all $g_\theta (Z^{(i)})$ by $g_\theta (Z^{(i)},Y^{(i)})$ in \Cref{table:WGANGD}.

In our simulation studies, we are particularly interested in how well $\hat g_n(Z)$ or $\hat g_n(Z,y)$ approximate the underlying true distribution of $X$ or $X$ given $Y = y$, respectively. We measure the approximation quality in our simulation studies with 
\begin{itemize}
    \item the empirical optimal transport distance, from now on OT, computed with the Python package POT \cite{pot}, between an equal amount of real samples $X_i$ and generated samples $\hat g_n(Z_j)$ as an estimator for $W_1(\IP^X,\IP^{\hat g_n(Z)})$ ($X_i$ given $Y_i = y$ and $\hat g_n(Z_j,y)$ in the conditional case),
    \item empirical $95\%$-confidence intervals (we use the abbreviation CI95) as discussed in \Cref{sec_uncon_conf_intervals}, \Cref{sec_asympt_conf_cwgan} and \Cref{sec_timeseries}, where we compute the empirical $2.5\%$- and $97.5\%$-quantiles of a statistic evaluated on $N$ generated samples.
\end{itemize}

\subsection{Synthetic data}

\textbf{Models: } To analyze the performance of the WGAN estimator in the unconditional case, we use the following model: We generate data
\begin{equation}
    X_i = g^{*}(Z_i), \quad i=1,...,n \label{uncon_sim_g}
\end{equation}
via the transformation
\begin{align*}
    g^*(z)=\big( \sin(z_1), \sin(z_2), &\sin(z_3), \exp(z_1),z_2^2+2z_3^3, \cos(2\pi z_1\cdot z_2\cdot z_3),\\
&z_1\cdot z_2\cdot z_3,(z_1+z_2+z_3)^2, z_1+z_2+z_3, 2x_1^4-x_2^3 \big),
\end{align*}
where $Z_i$ are i.i.d. uniformly distributed on $[0,1]^3$. That is, $d = 10$ and $d_Z = 3$. Here, $g$ is designed to contain a variety of smooth functions but also similarities between some of the coordinates.

%More precisely, we have $g^{*} \in \ssG(3,3,1,K)$ for some $K > 0$.

For the conditional WGAN estimator, we consider the following model: With $d=10,d_Z=7$ and $d_Y=3$, we simulate
\begin{equation}
    X_i = g^{*}_c(Z_i,Y_i),\quad i=1,...,n, \label{con_sim_g}
\end{equation}
where $g^*_c= g^*\circ h$ with
\[
h(z_1,\dots,z_7,y_1,y_2,y_3):=(z_1+z_2^2+z_3^3, z_4\cdot z_5 + z_6\cdot z_7, \sin(y_1)-y_2 \cdot y_{3}),
\]
and $(Z_i,Y_i)$, $i=1,...,n$ are i.i.d. uniformly distributed on $[0,1]^{10}$. Here, $g^{*}_c$ has a encoder-decoder structure according to Definition \ref{definition_c_autoencoder}.

For simplicity, we consider independent observations $Z_i,Y_i$ in both situations (that is, no serial correlation along $i=1,...,n$).

\textbf{Results:}  We examine the convergence behaviour of WGANs for increasing sample size $n\to\infty$ in \Cref{table:ntoinfty}. In the unconditional case, we construct confidence intervals for the statistic $T(X_1)$, where $T(x)=\sum_{j=1}^{10} x_j$, using $N=1000$ generated samples, by computing the empirical $2.5\%$- and $97.5\%$-quantiles of $\{T(\hat g (Z_j))\}_{j=1,\dots,N}$ (as in \reff{eq_uncon_conf_int}). In the conditional case, we approximate the statistic $T(X|Y=y)$ for $y=(0.5,0.5,0.5)$, using $N=1000$ generated samples, by computing the empirical $2.5\%$- and $97.5\%$-quantiles of $\{T(\hat g (Z_j,y))\}_{j=1,\dots,N}$ (as in \reff{conf_cwgan}). Note that the coverage of the constructed confidence intervals approaches $95\%$, while the optimal transport distance decreases. According to our results, the network sizes should grow with $n$, but we use a fixed architecture that performs well for all given $n$ to ensure comparability. Good coverage probabilities are already achieved for $n=960$ or $n=3200$, respectively. This highlights the fact that the WGAN is capable of detecting the sparse structure in the models \reff{uncon_sim_g} and \reff{con_sim_g} and realizes a faster convergence rate as announced in \Cref{theorem_wgan_excessbayes} and \Cref{theorem_wgan_conditional_excessbayes}.
\begin{table}[h!]
\centering
\begin{tabular}{ |c||c|c|c|c|c| }
\hline
 \multirow{2}{*}{\minitab[c]{Measured\\ quantity}} & \multicolumn{5}{|c|}{Number of samples} \\
 \cline{2-6}
  & 64 & 320 & 960 & 3200 & 9600\\
 \hline
 CI95, unc. &  47.92 (5.72)  & 52.26 (6.24) & 96.16 (1.18) & 94.50 (0.86) & \textbf{94.56} (0.84)\\
 OT, unc. & 1.634 (0.077)    & 1.630 (0.102) & 0.970 (0.130) & 0.412 (0.029) & \textbf{0.342} (0.026)\\
 \hline
 CI95, cond. & 24.96 (3.13) & 23.2 (1.67) & 45.32 (7.27) & 94.76 (1.93) & \textbf{94.78} (0.97)\\
 OT, cond. & 7.181 (0.187) & 6.720 (0.392) & 7.670 (0.307) & 1.967 (0.562) & \textbf{1.297} (0.341)\\
 \hline
\end{tabular}
\caption{Shows the coverage probability (in $\%$) of empirical $95\%$-confidence intervals for the sum of all components $T(x)=\sum_{j=1}^{10}x_j$ and the empirical optimal transport distance, each computed over $N=1000$ new i.i.d. samples, after 700 epochs of training. We train each model 5 times with different i.i.d. data.The left number denotes the mean over all runs, while the number in parentheses denotes the empirical standard deviation. We use a discriminator with 5 hidden layers of size 128 and a generator with 3 hidden layers of size 32.}
\label{table:ntoinfty}
\end{table}

\subsection{Real data application}
Practical approaches for time series generation using conditional GANs have been conducted by \cite{tcgan} using conditional vanilla GANs and \cite{tcwgan} using conditional WGANs. These works emphasize the potential of using generated data for data augmentation in other tasks, given small data sets (few shot learning).

For a real world simulation study we consider the mean temperatures $A_i \in \IR^d$, $i=1,...,n = 4779$ of $d=32$ German cities provided by the Deutscher Wetterdienst (German Metereological Service)\footnote{\url{https://opendata.dwd.de/climate_environment/CDC/observations_germany/climate/daily/kl/historical}}. Note that the chosen cities are spread throughout Germany, which can be seen in \Cref{fig:map}.
\begin{figure}
    \centering
    \includegraphics[width = 0.4\textwidth]{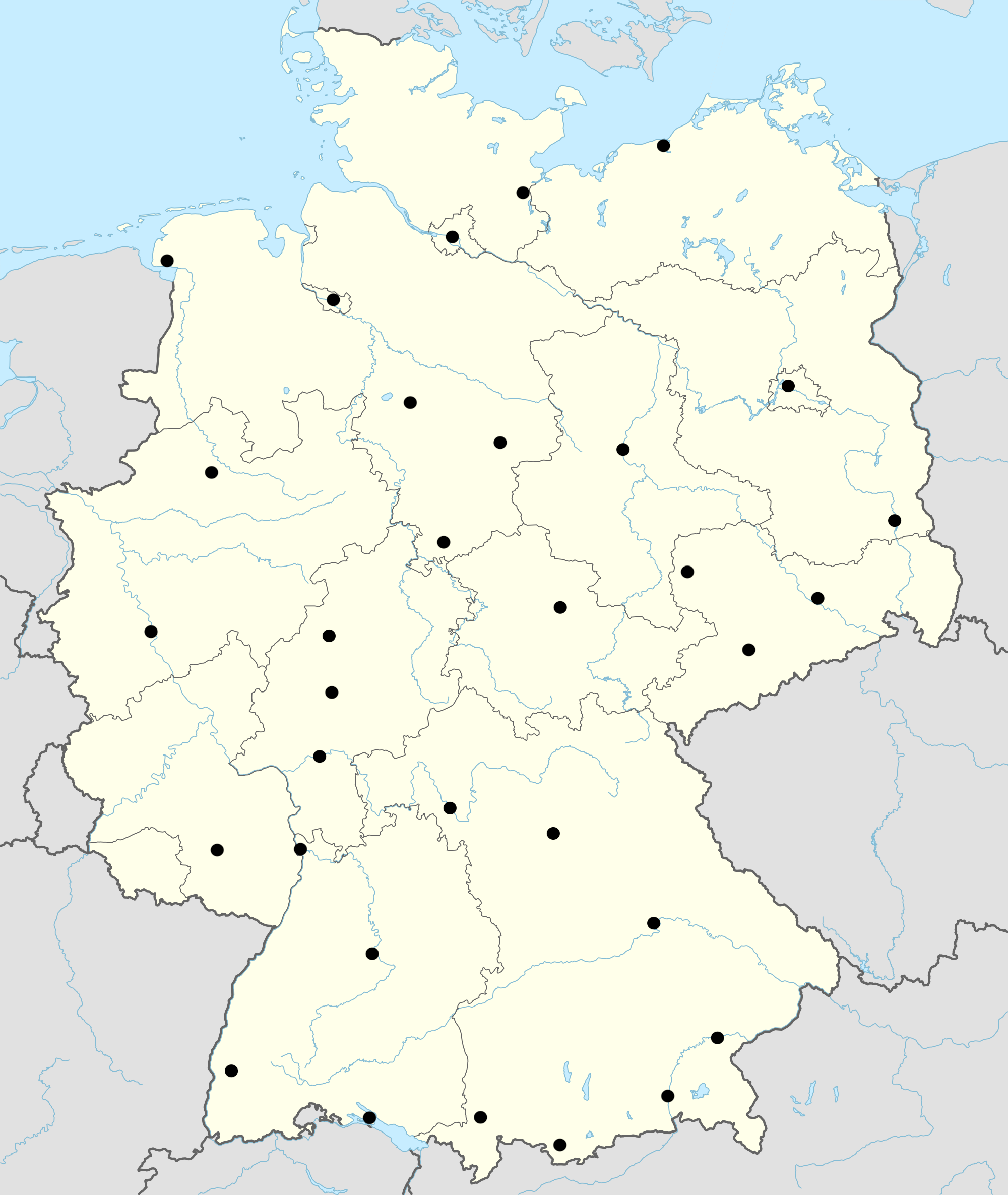}
    \caption{The authors of \cite{nathawut} collected weather data from the cities of Berlin, Braunschweig, Bremen, Chemnitz, Cottbus, Dresden, Erfurt, Frankfurt, Freiburg, Garmisch-Patenkirchen, Göttingen, Münster, Hamburg, Hannover, Kaiserslautern, Kempten, Köln, Konstanz, Leipzig, Lübeck, Magdeburg, Cölbe, Mühldorf, München, Nürnberg, Regensburg, Rosenheim, Rostock, Stuttgart, Würzburg, Emden and Mannheim.}
    \label{fig:map}
\end{figure}

In total we observe 4779 temperature values for each city over the period from 2006/07/01 to 2019/07/31. In the notation of \Cref{sec_timeseries}, given the temperatures of several cities of the previous day $A_{i-1}$ (so using only one lag $r=1$), we predict the temperature in Berlin $T(A_i)=A_{i1}$ of each day, i.e. $r=1$ and $T(x)=x_1$. The first day is not predicted. We use the first $n_{train}=4300$ days for training and the remaining $n_{test}=478$ days from 2018/04/10 to 2019/07/31 for testing.\\
We train cWGANs with 4-dimensional standard normal noise and use $A_{i-1}$ as conditional information. We train 3 different models.
\begin{enumerate}
    \item[(M1)] The first model only predicts the temperature in Berlin and $A_{i-1}$ only consists of the temperatures in Berlin, Braunschweig and Bremen of the previous day.
    \item[(M2)] The second model only predicts the temperature in Berlin but $A_{i-1}$ consists of the temperatures in all 32 cities of the previous day.
    \item[(M3)] The third model predicts the temperatures in all 32 cities and $A_{i-1}$ consists of the temperatures in all 32 cities of the previous day. The quality of the confidence intervals is only assessed for 1 city, namely Berlin.
\end{enumerate}

For all models we use the generators with 3 hidden layers with 10 neurons each and a discriminator with 5 hidden layers with 32 neurons each. \Cref{tab:temperatures} shows the progression over 1000 epochs of training. \Cref{table:temps} complements these illustrations with OT and CI95 values after 1000 epochs of training. The confidence intervals $I_{N,n}(A_{i-1})$ (cf. \reff{conf_cwgan}) for $T(A_i)$ are constructed from $N=1000$ realizations of $\IP^{Z}$ when $i$ belongs to the training set and from $N=10000$ realizations of $\IP^{Z}$ when $i$ belongs to the test set.

%These confidence intervals for each $(A_{i1},A_{i-1})$ consist of the empirical $2.5\%$- and $97.5\%$-quantiles after generating $1000$ samples for $A_{i-1}$ from the training set and $10000$ samples for $A_{i-1}$ from the test set (cf. \reff{conf_cwgan}). In the context of \Cref{sec_timeseries}, we use $r=1$, $T(x)=x_1$ and $I_{N,n}(Y_i)$.

%\iffalse

\afterpage{%
\thispagestyle{empty}
\begin{table}[ht]
\centering
\begin{tabular}{*{3}{m{0.33\textwidth}}}
\hline
\begin{center}(M1)\end{center} & \begin{center}(M2)\end{center} & \begin{center}(M3)\end{center}\\
\hline
\begin{center}\includegraphics[width = 0.33\textwidth]{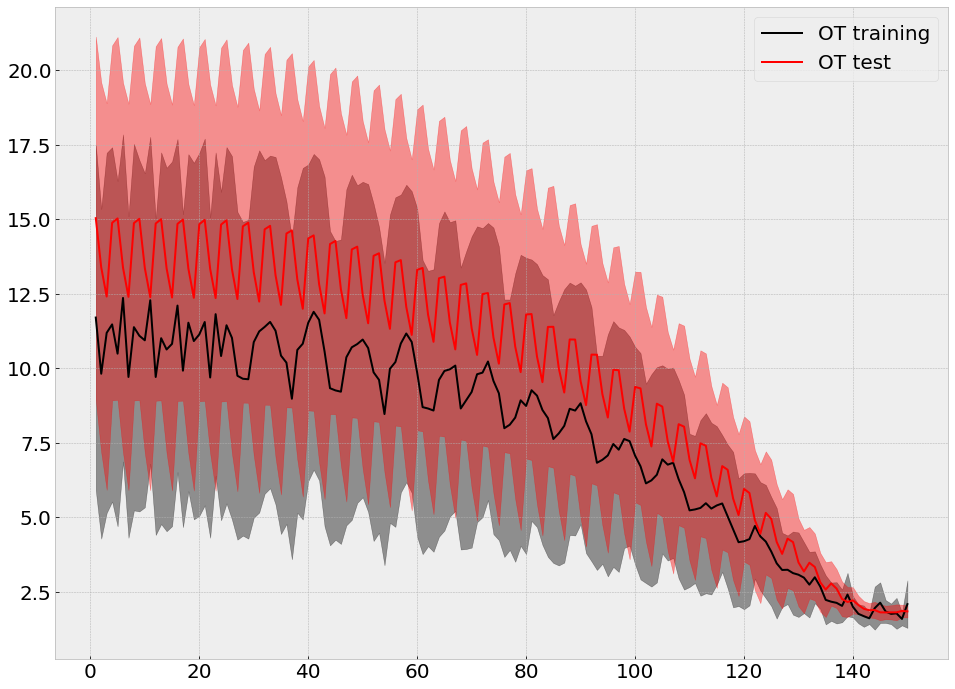}\end{center}
& \begin{center}\includegraphics[width = 0.33\textwidth]{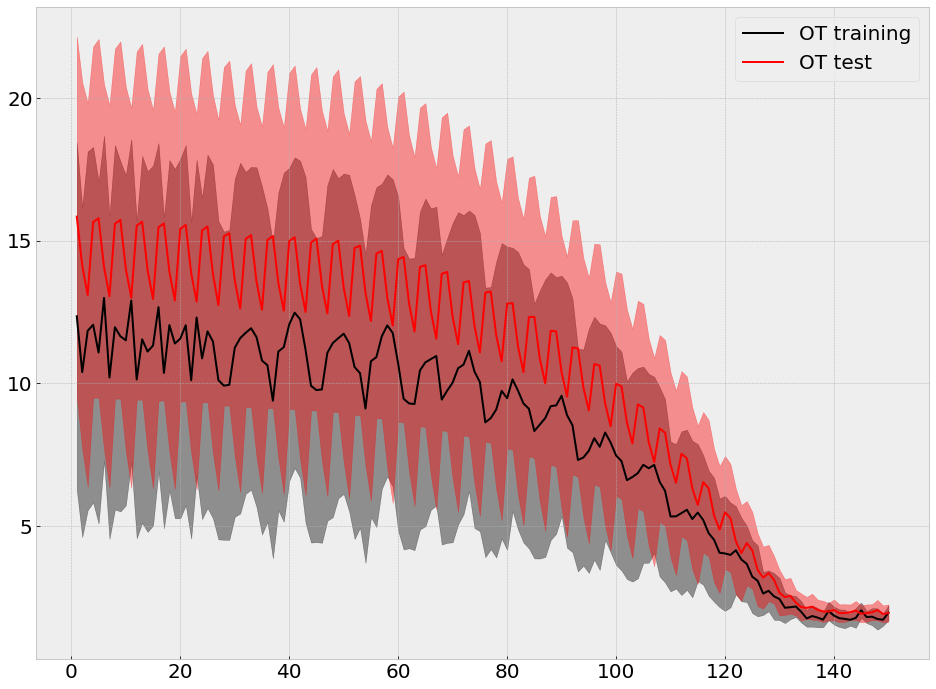}\end{center}
& \begin{center}\includegraphics[width = 0.33\textwidth]{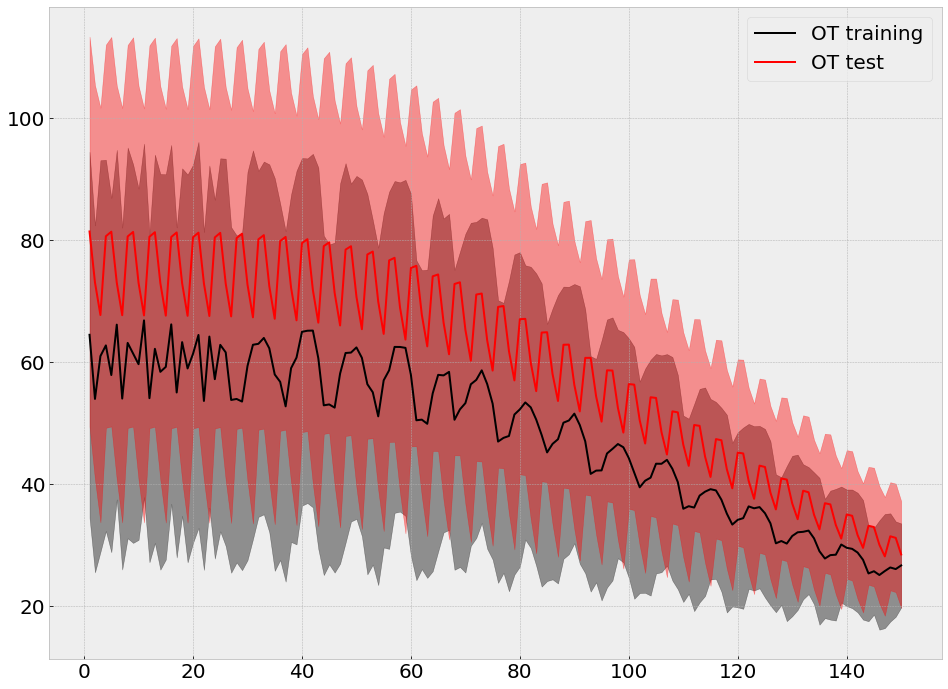}\end{center}\\
\hline
\begin{center}\includegraphics[width = 0.33\textwidth]{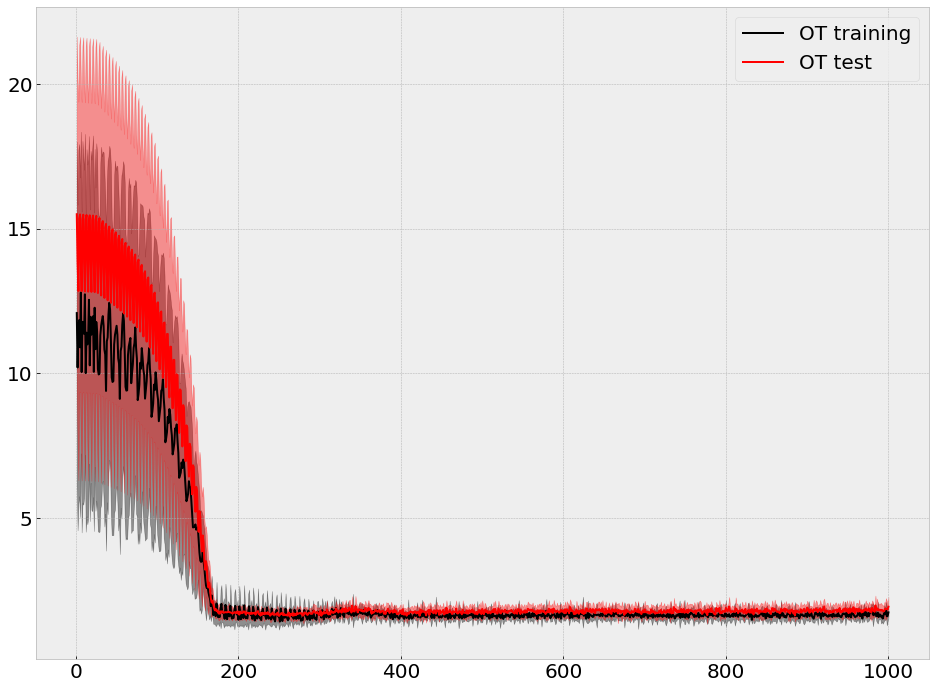}\end{center}
& \begin{center}\includegraphics[width = 0.33\textwidth]{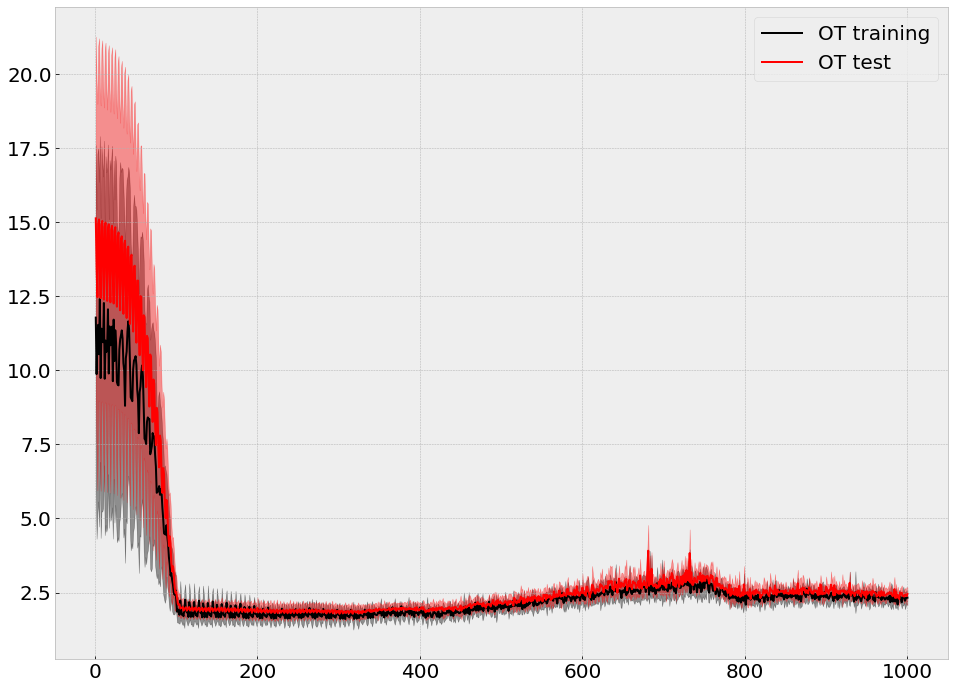}\end{center}
& \begin{center}\includegraphics[width = 0.33\textwidth]{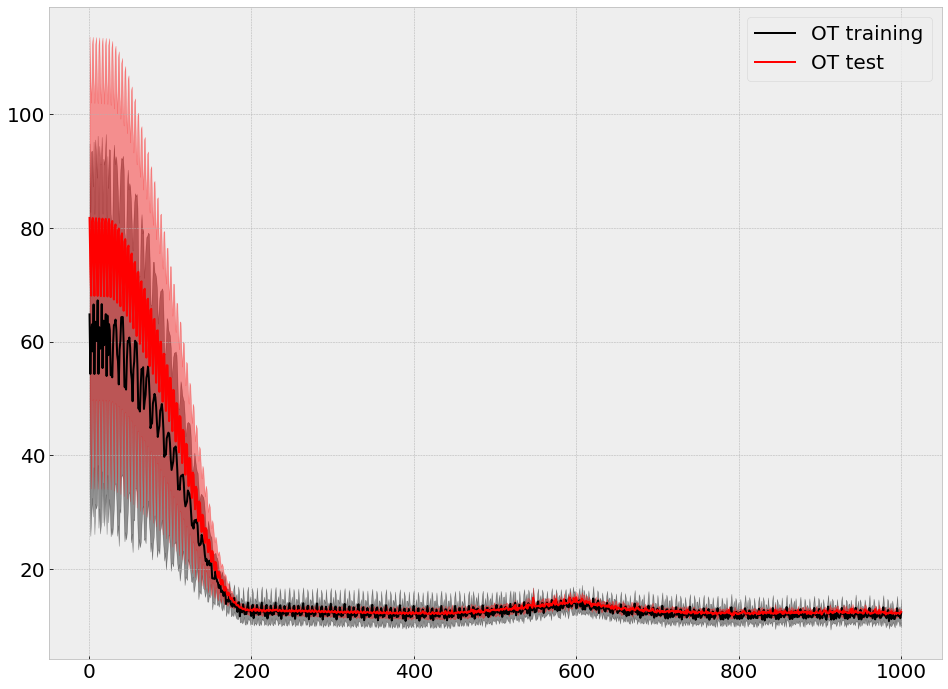}\end{center}\\
\hline
\begin{center}\includegraphics[width = 0.33\textwidth]{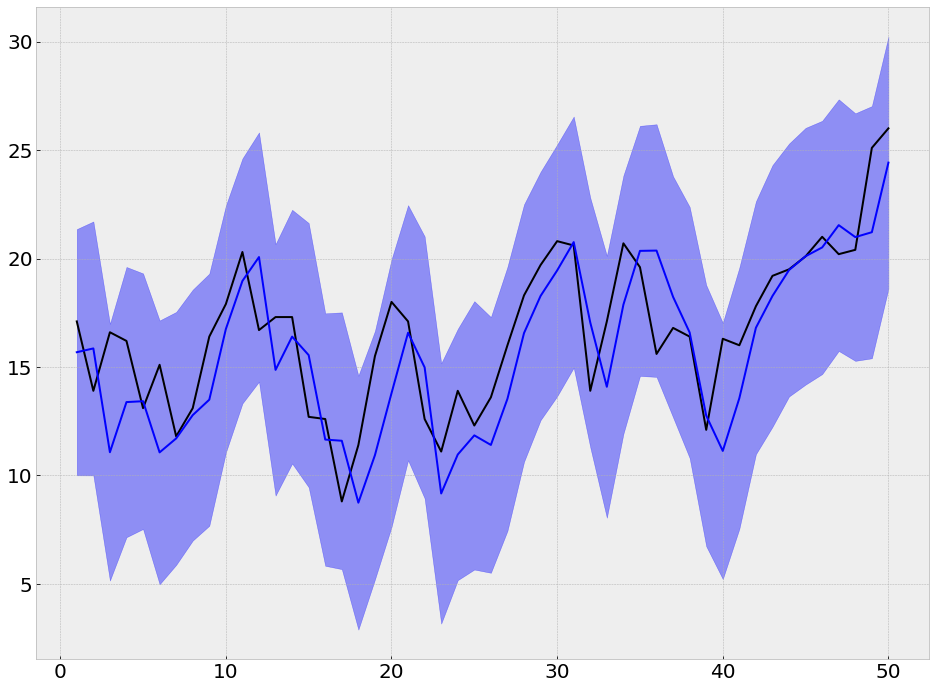}\end{center}
& \begin{center}\includegraphics[width = 0.33\textwidth]{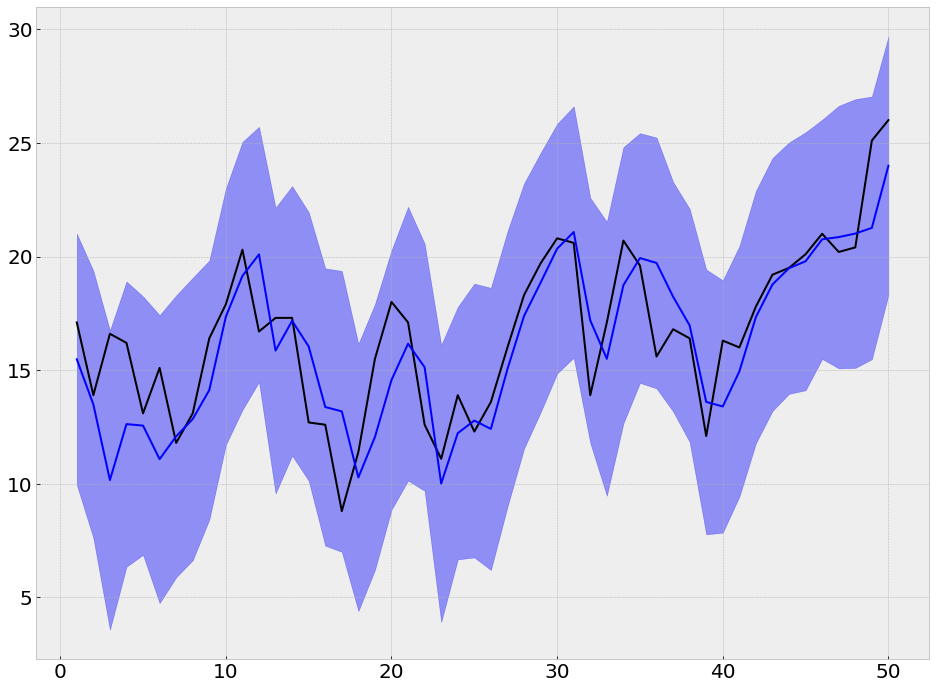}\end{center}
& \begin{center}\includegraphics[width = 0.33\textwidth]{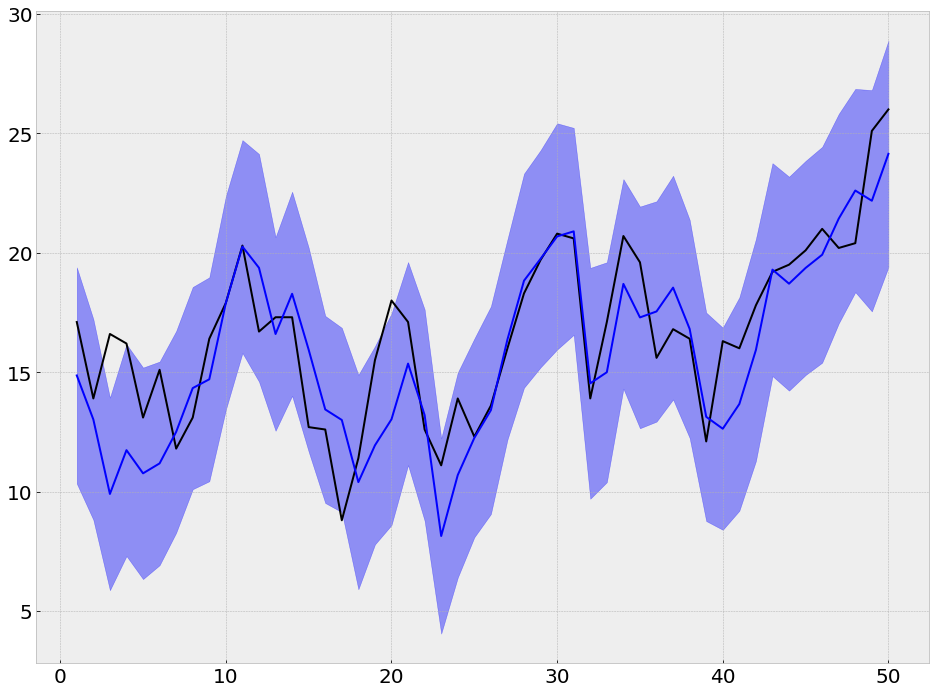}\end{center}\\
\hline
\begin{center}\includegraphics[width = 0.33\textwidth]{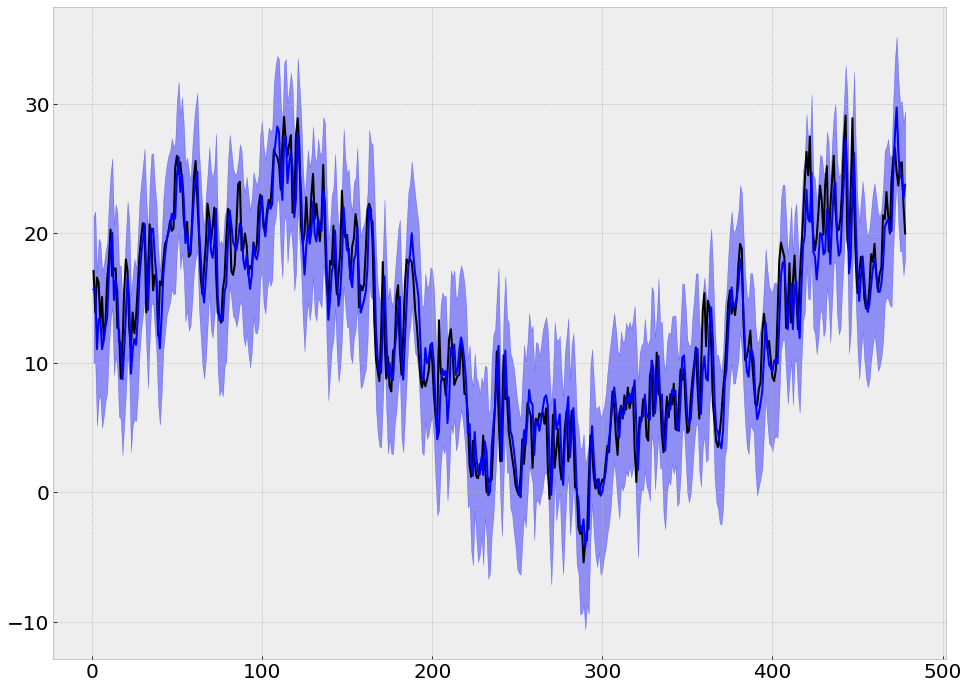}\end{center}
& \begin{center}\includegraphics[width = 0.33\textwidth]{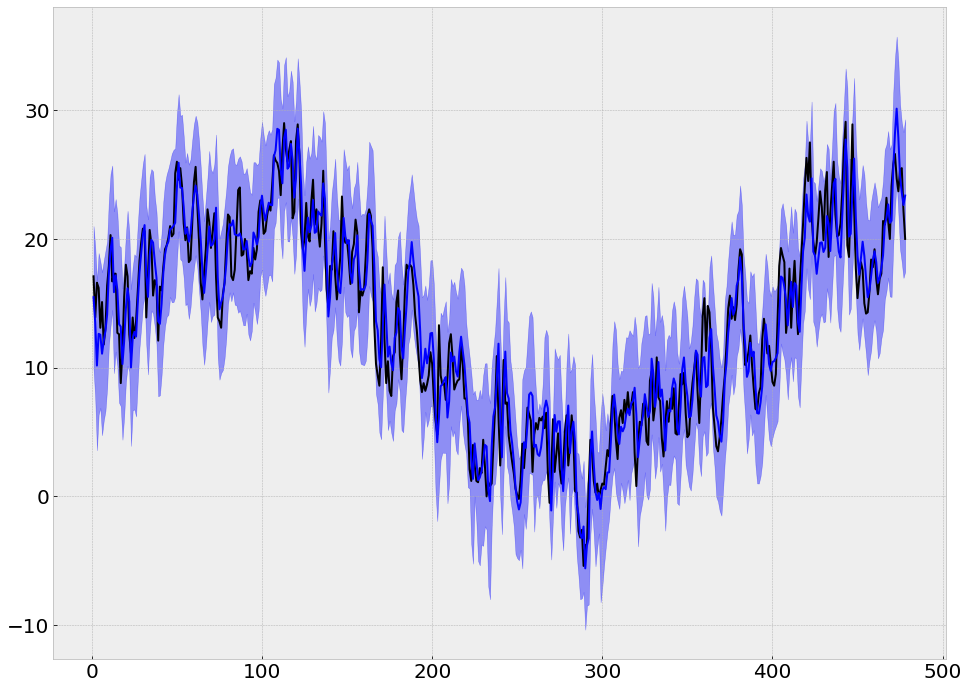}\end{center}
& \begin{center}\includegraphics[width = 0.33\textwidth]{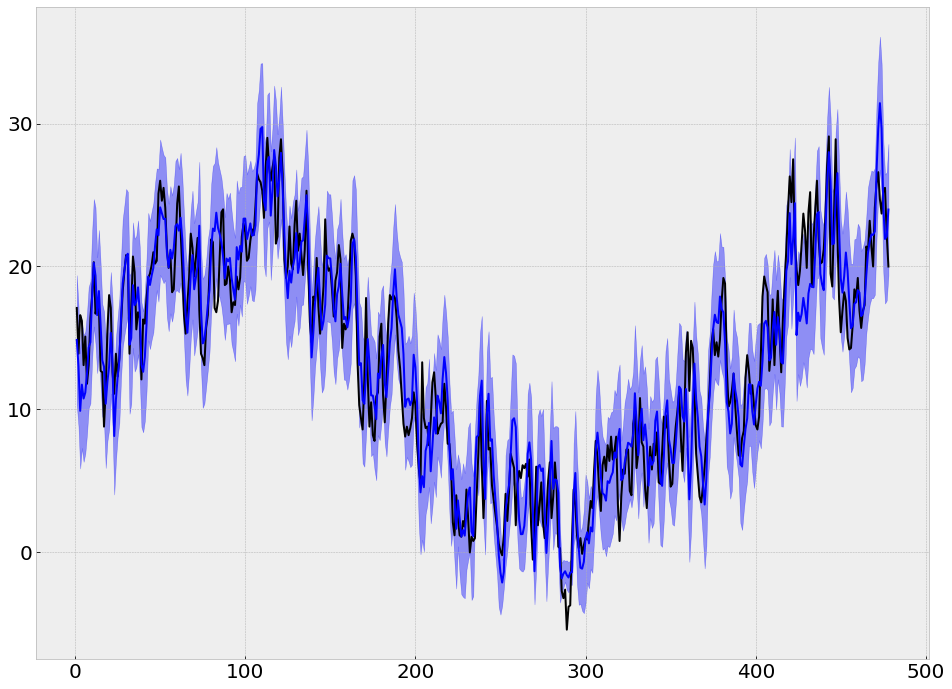}\end{center}\\
\hline
\end{tabular}
\caption{Training process of a cWGAN depicted after 1000 epochs over temperature data for the models (M1)-(M3).\\
The first two rows show the estimated optimal transport distance between real and generated data evaluated 10 times on 1 batch from the training set (black) and test set (red), over 150 and 1000 epochs respectively. Depicted is the mean and a $\pm\sigma$-confidence band. The other two rows show the actual mean temperatures in Berlin (black) and the average generated prediction (blue) with a $\pm 3\sigma$-confidence band over 1000 generated points per time step, for the first 50 days of the test set and the whole test set respectively.}
\label{tab:temperatures}
\end{table}
\clearpage
}
%\fi

The optimal transport distance is computed jointly over the $A_{i-1}$ and the predicted/real temperatures in Berlin and still has high variance as the dimension is $4$ for (M1) and $33$ for (M2) and (M3). Since in the first model $A_{i-1}$ is only 3-dimensional, the optimal transport distance is only comparable between the second and third model.

\begin{table}[h]
\centering
\begin{tabular}{|l||c c c|} 
 \hline
  & Three to Berlin & All to Berlin & All to All \\ [0.5ex] 
 \hline
 Test OT, 1000 epochs & 1.39 & 2.38 & 2.37 \\
 \hline
 Empirical $95\%$-confidence, train & 92.60\% & 91.14\% & 74.02\% \\
 Empirical $95\%$-confidence, test & 89.96\% & 89.54\% & 70.71\% \\ [1ex] 
 \hline
\end{tabular}
\caption{The first row computes the empirical optimal transport distance (OT) on the whole test set. Rows 2 and 3 show the respective proportion of $T(A_i)$ lying in the intervals $I_{N,n}(A_{i-1})$ between the empirical $2.5\%$- and $97.5\%$-quantiles, for the training and test set, respectively.}
\label{table:temps}
\end{table}

Note that in all scenarios, the generator firstly learns the predictions with overconfidence. For (M1), the training data lies much more densely in the $A_{i-1}$-space. (M1) learns the distribution much faster, but eventually the (M2)-model achieves comparable performance. The most complex (M3)-model takes the longest to converge but already performs decently considering that it performs a $32$-dimensional prediction with the same small generator architecture. To have comparable results, we used 1000 training epochs for the estimators in \Cref{table:temps}. However, the quality of the estimators may still increase for more training epochs. For instance, after 2000 epochs of training, we achieve 2.35 OT and $86.19\%$ coverage on the test set for the temperature in the city Berlin in model (M3).  Overall, the results are quite satisfying and motivate that the cWGAN estimator is able to find some sparse underlying structures in the data.

\section{Conclusion}
\label{sec_conclusion}

To our knowledge, this paper is the first where  convergence rates for the excess Bayes risk of Wasserstein GANs and conditional Wasserstein GANs are derived under structural assumptions on the space of generators.

We have formalized the empirical WGAN objective with growing critic networks and have shown that this objective still metrizes weak convergence. Our results yield recommendations on the size of generator networks and unveil the potential use of conditional WGANs in high-dimensional time series forecasting, in particular the construction of confidence intervals. All our results hold for dependent data, where the dependence is measured with absolutely regular $\beta$-mixing. Both our synthetic and real world simulations demonstrate good empirical coverage for confidence intervals in multidimensional applications.

Additionally, we have included a first approach to formalize the availability of a growing number of observations when training is performed with multiple epochs. The corresponding result justifies the use of very large generator networks without suffering from slow convergence rates. Our attempt could be explored in other contexts and extended to the conditional case. In future work, one could also try to study the convergence behaviour of local instead of global minimizers of the empirical WGAN objectives. Furthermore, it would be interesting to include the gradient penalty in the theoretical results and use other GAN losses or network architectures, such as the Groupsort activation function \cite{anil2018sorting} for the critic. It would be interesting to refine the approximation results from \cite{sh} to gain more insight into the theoretical properties of our modified Wasserstein distance $W_{1,n}$.

\bibliographystyle{plain}
\bibliography{refs}

\appendix

%\section{Proofs}

\section{Proofs of Section \ref{sec_results}}

\begin{proof}[Proof of Lemma \ref{lemma_weak_convergence}] 
    By Lemma \ref{lemma_w1beta_approx},
    \[
        \IE W_1^{\gamma}(\hat g_n) \le \IE W_{1,n}(\hat g_n) + C \phi_n \to 0.
    \]
    Let $f\in C^{\gamma}([0,1]^d,1)$ be arbitrary. By independence of $Z$ and $\hat g_n(\cdot)$, it follows that $\IE[\IE[f(g(Z))]\big|_{g=\hat g_n}] = \IE[\IE[f(\hat g_n(Z))|\hat g_n]] = \IE f(\hat g_n(Z))$. Thus
    \begin{equation}
        |\IE f(X) - \IE f(\hat g_n(Z))| \le \IE W_1^{\gamma}(\hat g_n) \to 0.\label{lemma_weak_convergence_eq1}
    \end{equation}
    We now show weak convergence $\hat g_n(Z) \overset{d}{\to} X$. Let $A \subset [0,1]^d$ be a closed set (and thus compact). Let $\varepsilon > 0$. Define $\rho_{\varepsilon,A}(x) := 1 - \varphi(\frac{d(x,A)}{\varepsilon})$, where $d(x,A) := \inf_{y\in A}\|x-y\|_{\infty}$ and $\varphi:\IR \to [0,1]$ is an arbitrary infinitely differentiable function with $\varphi(x) = 0$ for $x \le 0$ and $\varphi(x) = 1$ for $x \ge 1$, for instance one may define $\varphi(x) = e^{-1/x}\cdot (e^{-1/x} + e^{-1/(1-x)})^{-1}$ for $x\in [0,1]$.
    
    Note that $\Ii_A(x) \le \rho_{\varepsilon,A}(x)$. Furthermore, for $d(\varepsilon) > 0$ small enough, $d(\varepsilon)\cdot \rho_{\varepsilon,A} \in C^{\gamma}([0,1]^d,1)$. By these arguments and \reff{lemma_weak_convergence_eq1},
    \begin{eqnarray*}
        \IP(\hat g_n(Z) \in A) &\le& \IE \rho_{\varepsilon,A}(\hat g_n(Z)) = d(\varepsilon)^{-1}\IE[d(\varepsilon)\rho_{\varepsilon,A}(\hat g_n(Z))]\\
        &\to& d(\varepsilon)^{-1}\IE[d(\varepsilon)\rho_{\varepsilon,A}(X)] = \IE \rho_{\varepsilon,A}(X).
    \end{eqnarray*}
    We conclude that $\limsup_{n\to\infty}\IP(\hat g_n(Z) \in A) \le \IE \rho_{\varepsilon,A}(X)$. Since $\rho_{\varepsilon,A}(x) \to \Ii_A(x)$ for $\varepsilon \to 0$, the dominated convergence theorem implies $\limsup_{n\to\infty}\IP(\hat g_n(Z) \in A) \le \IP(X \in A)$. The result now follows from the portmanteau lemma for weak convergence.
\end{proof}

\begin{proof}[Proof of Lemma \ref{lemma_weak_convergence_distribution}]
    Fix $x\in\IR$. By the law of large numbers (applied conditionally on $\hat g_n$), we have that almost surely, for $N\to\infty$,
    \begin{equation}
        \hat F_{N,n}(x) \to \IP(\hat g_n(Z) \le x|\hat g_n).\label{lemma_weak_convergence_distribution_eq1}
    \end{equation}
    We now conduct a similar argumentation as in the proof of Lemma \ref{lemma_weak_convergence} based on the stochastic convergence $W_{1}^{\gamma}(\hat g_n) \overset{p}{\to} 0$. Let $A \subset [0,1]$ be a closed subset. Then
    \begin{eqnarray*}
        &&\IP\big(\IP(\hat g_n(Z) \in A|\hat g_n) - \IP(X \in A) \ge \rho\big)\\
        &\le& \IP\big(\IE[d(\varepsilon)\rho_{\varepsilon,A}(\hat g_n(Z))|\hat g_n] - \IE[d(\varepsilon) \rho_{\varepsilon,A}(X)] \ge \frac{\rho}{2} d(\varepsilon)\big)\\
        &&\quad\quad + \IP(\IE \rho_{\varepsilon,A}(X) - \IP(X \in A) \ge \frac{\rho}{2}\big).
    \end{eqnarray*}
    While the second summand is 0 for $\varepsilon > 0$ small enough, the first summand tends to zero by $W_{1}^{\gamma}(\hat g_n) \overset{p}{\to} 0$. This shows that
    \[
        \lim_{n\to\infty}\IP\big(\IP(\hat g_n(Z) \in A|\hat g_n) - \IP(X \in A) \ge \rho\big) = 0.
    \]
    Using typical proof strategies from the portemanteau lemma, we first see that for any open subset $U \subset [0,1]$, and any $\rho > 0$,
    \[
        \lim_{n\to\infty}\IP\big(\IP(\hat g_n(Z) \in U|\hat g_n) - \IP(X \in U) \le -\rho\big) = 0.
    \]
    and for any $x\in[0,1]$ which is a continuity point of $F_X$ and any $\rho > 0$,
    \begin{equation}
        \lim_{n\to\infty}\IP\big(\big|\IP(\hat g_n(Z) \le x|\hat g_n) - F_X(x)\big| \ge \rho\big) = 0.\label{lemma_weak_convergence_distribution_eq2}
    \end{equation}
    From \reff{lemma_weak_convergence_distribution_eq1} and \reff{lemma_weak_convergence_distribution_eq2} we obtain that for $\rho > 0$ and any $x \in \IR$,
    \begin{eqnarray*}
        &&\limsup_{n\to\infty}\limsup_{N\to\infty}\IP\big(\big|\hat F_{N,n}(x)  - F_X(x)\big| \ge \rho\big)\\
        &\le& \limsup_{N\to\infty}\IP\big(\big|\hat F_{N,n}(x)  - \IP(\hat g_n(Z) \le x|\hat g_n)\big| \ge \frac{\rho}{2}\big)\\
        &&\quad\quad + \limsup_{n\to\infty}\IP\big(\big|\IP(\hat g_n(Z) \le x|\hat g_n)  - F_X(x)\big| \ge \frac{\rho}{2}\big) = 0.
    \end{eqnarray*}
    By continuity of $F_X$, standard decomposition arguments from the Polya theorem about uniform convergence of distribution functions provide
    \[
        \limsup_{n\to\infty}\limsup_{N\to\infty}\IP\big(\sup_{x\in[0,1]}\big|\hat F_{N,n}(x)  - F_X(x)\big| \ge \rho\big) = 0.
    \]
    The result of the lemma now follows for plugging in $x = X$.
\end{proof}

\section{Error Decomposition}\label{sec:err}

\subsection{Unconditional WGAN: Basic inequality}
We abbreviate $\gm = \ssG(d_Z,d_g,\beta,K)$ and $\rgt := \mathcal{R}(L_g,\mathbf{p}_g,s_g)$, $\rft = \mathcal{R}(L_f,\mathbf{p}_f,s_f)$.

Recall from \reff{wgan_unconditional} that
\[
    \hat g_n = \argmin_{g \in \rgt}\hat W_{1,n}(g).
\]

\begin{prop}[WGAN: Basic inequality]\label{basic_inequality_wgan}
    It holds that
    \[
        \wn(\hat g_n) - \inf_{g\in\gm} \wn(g) \le \sqrt{d}\cdot A_n + 2\cdot E_n,
    \]
    where
    \begin{eqnarray}
        A_n &:=& \sup_{g \in \gm}\inf_{\tilde g\in \rg}\|g - \tilde g\|_{\infty},\label{wgan_approx_error}\\
        E_n &:=& \sup_{f\in \rft} |(\hat\IP_n^X - \IP^X) f| +  \sup_{g\in \rgt, f\in \rft} |(\hat\IP_{n\Epo}^Z-\IP^Z)(f\circ g)|.\label{wgan_est_error}
    \end{eqnarray}
\end{prop}
\begin{proof}[Proof of Proposition \ref{basic_inequality_wgan}]

First, we have
\[
\wn(\hat g_n)-\inf_{g\in\gm} \wn(g)\le e_n + a_n,
\]
where
\[
    e_n := \wn(\hat g_n) - \inf_{g\in\rgt} \wn(g)
\]
is the estimation error and
\[
    a_n := \inf_{g\in \rgt}\wn(g) - \inf_{g\in \gm}\wn(g)
\]
is the approximation error. Note that $A_n$ is upper bounded (not in absolute value!) as follows:
\begin{equation}
    a_n \le \sup_{g \in \gm}\inf_{\tilde g\in \rg}|W_{1,n}(g) - W_{1,n}(\tilde g)|.\label{basic_inequality_eq1}
\end{equation}
Since all functions $f$ in the supremum in $W_{1,n}$ satisfy $\|f\|_L \le 1$, we have
\begin{eqnarray*}
    && |W_{1,n}(g) - W_{1,n}(\tilde g)|\\
    &\le& \big|\sup_{f\in \rft,\|f\|_L\le 1}\big\{\IE f(X) - \IE f(g(Z))\big\} - \sup_{f\in \rft, \|f\|_L \le 1}\{\IE f(X) - \IE f(\tilde g(Z))\big\}\big|\\
    &\le& \sup_{f \in \rft, \|f\|_L\le 1} \IE|f(g(Z)) - f(\tilde g(Z))|\\
    &\le& \sqrt{d}\|g - \tilde g\|_{\infty}.
\end{eqnarray*}
We conclude from \reff{basic_inequality_eq1} that
\[
    a_n \le \sqrt{d}\inf_{g\in \rg}\sup_{\tilde g \in \gm}\|g - \tilde g\|_{\infty}.
\]
We now investigate the estimation error $E_n$. Let $\varepsilon > 0$. Then there exists $g^{*} \in \rgt$ with $\inf_{g \in \gm}\wn(g) \le \wn(g^{*}) + \varepsilon$. We obtain
\begin{equation}
    e_n = \wn(\hat g_n)-\inf_{g\in\gm} \wn(g) \le  \wn(\hat g_n)-\wn(g^{*}) \col{+}\co{-} \varepsilon.\label{basic_inequality_eq2}
\end{equation}
In order to bound $\wn(\hat g_n) - \wn(g^{*})$, note that by the minimization property of $\hat g_n$, 
\begin{eqnarray*}
    &&\wn(\hat g_n) - \wn(g^{*})\\
    &=& \ew(\hat g_n) - \ew(g^{*})\\
    &&\quad\quad -\Big(\{\ew(\eg) - \wn(\eg)\} - \{\ew(\gn) - \wn(\gn)\}\Big)\\
    &\le& 2\sup_{g\in\rgt} \big| \ew(g)-\wn(g) \big|.
\end{eqnarray*}
Letting $\varepsilon \downarrow 0$, we obtain from \reff{basic_inequality_eq2} that
\begin{equation}
    e_n \le 2\sup_{g\in\rgt} \big| \ew(g)-\wn(g) \big|.\label{basic_inequality_eq3}
\end{equation}
Note that
\begin{eqnarray*}
   &&\sup_{g\in \rgt} |\ew(g) - \wn(g)|\\
   &\le& \sup_{g\in \rgt} \Big| \sup_{f\in \rft} \big\{ (\hat\IP_n^X - \IP^X) f - (\hat\IP_{n\Epo}^Z-\IP^Z)(f\circ g)\big\}\Big|\notag\\
    &\le& \sup_{f\in \rft} |(\hat\IP_n^X - \IP^X) f| +  \sup_{g\in \rgt, f\in \rft} |(\hat\IP_{n\Epo}^Z-\IP^Z)(f\circ g)|
\end{eqnarray*}
Insertion into \reff{basic_inequality_eq3} yields the assertion.
\end{proof}

\subsection{Approximation error}

\label{sec:apprerror}
To bound the approximation error $A_n$, we use the approximation theory from \cite{sh} and statements about the Lipschitz constant in \cite{nathawut}.

\begin{thm}[\cite{sh}, Theorem 5 and \cite{nathawut}, Theorem 9.14]\label{reluapprox}
For all
\[
h\in C^\beta ([0,1]^r,K),\quad k\geq 1 \quad\text{and}\quad N\geq (\beta+1)^r\vee (K+1) e^r,
\]
there exists a network \[
\tilde h \in \relu\big(L, (r,6(r+\lceil\beta\rceil)N,\dots,6(r+\lceil\beta\rceil)N,1),s,\infty\big) \]
with \[
L=8+(k+5)(1+\lceil \log_2(r\vee\beta)\rceil) \quad\text{and}\quad s\le 141(r+\beta+1)^{3+r} N(k+6), \]
such that,
\[
\| h - \tilde h\|_{L^\infty([0,1]^r)}\le (2K+1)(1+r^2+\beta^2)6^r N 2^{-k}+K 3^\beta N^{-\beta/r}.
\]
Furthermore, $\tilde h$ satisfies for any $x,y \in [0,1]^r$ that
\[
    |\tilde h(x) - \tilde h(y)| \le \mathrm{Lip}(N,k) \cdot |x-y|_{\infty},
\]
where
\[
    \mathrm{Lip}(N,k) := 2\beta F(K+1)e^r( 24 r^6 2^r N 2^{-k} + 3r).
\]
\end{thm}

\begin{lem}\label{lem:approx}
    Let $\beta \ge 1$, $d_g \in\IN, \Epo\in\IN$. Let $ N\geq (\beta+1)^{d_g}\vee (K+1) e^{d_g}$.\\
    If $\rg = \mathcal{R}(L_g,\mathbf{p}_g,s_g)$ satisfies $F \ge K \vee 1$ and
    \[
        L_g\ge \log_2(n\Epo) \log_2(4d_g\vee 4\beta), \quad \min_{i=1,...,L_g}p_i\gtrsim dN \quad\text{and}\quad s_g\gtrsim d N \log_2(n\Epo),
    \]
    where the bounding constants only depend on $\beta,d_g$, then $A_n$ from \reff{wgan_approx_error} satisfies that for $n$ large enough,
    \[
        A_n \lesssim \frac{N}{n\Epo} + N^{-\beta/d_g},
    \]
    where the bounding constants only depend on $\beta,d_g$ and $K$.
\end{lem}
\begin{proof}[Proof of Lemma \ref{lem:approx}]
Given $g\in\ssG$, we can write $g_i\in C^{\beta}([0,1]^{d_g}, K)$, since each component function only depends on $d_g$ arguments. Applying \Cref{reluapprox} with $k = \lceil\log_2(n\Epo)\rceil$ to each component function yields that there exists a $\tilde g$ in the class
\[
\relu\left( L, (d_g, 6(d_g+\lceil\beta\rceil) N,\dots,  6(d_g+\lceil\beta\rceil) N,1), s_g,\infty\right),
\]
such that $||g_i-\tilde g_i||_\infty \lesssim N 2^{-k} + N^{-\beta/d_g}$, where $L = k \lceil\log_2(4d_g\vee 4\beta) \rceil,$ $s\lesssim N k$ and the bounding constants only depend on $K,d_g,\beta$.\\
Thus a network computing all $\tilde g:=(\tilde g_i)_{i=1,\dots,d}$ in parallel lies in the class \[
\relu\left( L, (d_g, 6 d (d_g+\lceil\beta\rceil) N,\dots,6 d (d_g+\lceil\beta\rceil) N, d), ds,\infty\right),
\]
and it holds that \begin{align}
    || g- \tilde g ||_\infty \lesssim N 2^{-k} + N^{-\beta/d_g}.\label{approx_bound}
\end{align}
$\tilde g$ may not satisfy $\|\tilde g\|_{\infty} \le F$. However, $\tilde g^{\circ} := (\frac{\|g\|_{\infty}}{\|\tilde g\|_{\infty}}\wedge 1)\tilde g$ still fulfills $\tilde g^{\circ} \in \relu( L,  p, s)$ and $\|\tilde g^{\circ}\|_{\infty} \le \|g\|_{\infty} \le K \le F$. Due to $\|\tilde g^{\circ} - g\|_{\infty} \le 2\|\tilde g - g\|_{\infty}$, \reff{approx_bound} still holds for $\tilde g^{\circ}$ with changed constants. 
\end{proof}

\subsection{Estimation error}

To upper bound the entropy bracketing numbers of the neural network sets $\mathcal{R}(L,\mathbf{p},s)$, we use the following Lemma taken from \cite{sh}.

For a class $\ssF \subset \{f:\IR^r \to \IR \text{ measurable}\}$ and some norm $\|\cdot\|$ on $\ssF$, we denote by $N_{[]}(\varepsilon, \ssF, \|\cdot\|)$ the number of $\varepsilon$-brackets which are needed to cover $\ssF$. Here, an $\varepsilon$-bracket $[l,u]$ is a set $[l,u] = \{f\in \ssF| \forall x\in \IR^r: l(x) \le f(x) \le u(x)\}$ such that $\|u - l\| \le \varepsilon$.

The bracketing entropy integral of $\ssF$ with respect to $\|\cdot\|$ is denoted by
\[
    J_{[]}(\delta, \ssF, \|\cdot\|) = \int_0^{\delta}\sqrt{1+\log N_{[]}(\varepsilon,\ssF, \|\cdot\|)} \dif \varepsilon.
\]
The covering numbers $N(\varepsilon, \ssF, \|\cdot\|)$ denote the least number of elements $v_1,...,v_m \in \ssF$ such that $\ssF \subset \bigcup_{j=1}^{m}\{y \in \ssF:\|y-v_j\| < \varepsilon\}$. Accordingly, we define the covering entropy integral $J(\delta, \ssF, \|\cdot\|) = \int_0^{\delta}\sqrt{1+\log N(\varepsilon,\ssF, \|\cdot\|)} \dif \varepsilon$. We need both bracketing and covering numbers since the approximation results in \cite{sh} were defined in terms of covering numbers while the empirical process results of \cite{dl} are in terms of bracketing numbers. However, there is a simple connection via
\begin{equation}
    N_{[]}(\delta,\ssF,\|\cdot\|) \le N(\frac{\delta}{2},\ssF,\|\cdot\|).\label{bracket_covering}
\end{equation}

For mixing coefficients $\beta_X(k)$, $k\in\IN_0$, \cite{betanorm} defined the $\|f\|_{2,\beta}$-norm as follows: Let $\beta_X^{-1}$ be the cadlag inverse of $\beta_X(t)=\beta(\lfloor t \rfloor)$ for $t\geq 1$ and $\beta_X(t)=1$ otherwise. Let $Q_f$ be the inverse of the tail function $t\mapsto\P(|f(X_1)|>t)$. Define 
\[
    \|f\|_{2,\beta} = \Big( \int_0^1 \beta_X^{-1} (u) Q_f(u)^2 du \Big)^{1/2}.
\]
In \cite[Lemma 1]{betanorm} it is stated that for $B := \sum_{k=0}^{\infty}\beta_X(k)$, one has
\begin{equation}
    \|f\|_{2,\beta}\le B^{1/2}\cdot \|f\|_{\infty}.\label{standardbound_betnorm}
\end{equation}

\begin{lem}[\cite{sh}, Lemma 5] \label{shentr}
For all $\delta>0,$ it holds that
\[
\log\Big( N\big(\delta,\relu(L,\mathbf{p},s,\infty),\|\cdot\|_\infty\big)\Big)\le (s+1) \log\Big( 2 \frac{L+1}{\delta} \big(\prod_{l=0}^{L+1} (p_l +1)\big)^2 \Big).
\]
\end{lem}

In the following, we use the following abbreviation
\begin{equation}
    \gamma(L,\mathbf{p},s):=2(s+1)\log\big(4(L+1)\prod_{l=0}^{L+1} (p_l+1)\big).\label{definition_groesse_netzwerk}
\end{equation}

The following lemma is the basic result we use to bound the estimation error both in expectation and with high probability. It makes use of maximal inequalities and large deviation bounds derived in Section \ref{sec:entropybound} for mixing sequences.

\begin{lem}\label{lem:ent}
Suppose that there exist constants $\kappa > 1, \alpha > 1$ such that for all $k\in\IN$, $\beta_X(k) \le \kappa\cdot k^{-\alpha}$. Suppose that
\[
    \gamma(L,\mathbf{p},s) \le n.
\]
Then there exists a constant $C>0$ only depending on characteristics of $(X_i)$ and $B,F,\kappa,\alpha$ such that
\begin{equation}
\mathds{E}^* \sup_{f\in \relu(L,\mathbf{p},s)}|(\hat \IP_n^X-\IP^X)f|\le C\cdot \Big(\frac{\gamma(L,\mathbf{p},s)}{n}\Big)^{1/2}.\label{lem:ent_eq1}
\end{equation}
Furthermore, with probability at least $1 - 2n^{-1} - (\frac{\log(n)}{n})^{\frac{\alpha-1}{2}}$ and a different constant $C > 0$ depending on the same quantities,
\begin{equation}
    \sup_{f\in \relu(L,\mathbf{p},s)}|(\hat \IP_n^X-\IP^X)f| \le C\cdot \Big[\Big(\frac{\gamma(L,\mathbf{p},s)}{n}\Big)^{1/2} + \Big(\frac{\log(n)}{n}\Big)^{1/2}\Big].\label{lem:ent_eq2}
\end{equation}
\end{lem}
\begin{proof}[Proof of Lemma \ref{lem:ent}]
We abbreviate $\relu=\relu(L,\mathbf{p},s,F)$. Using \Cref{shentr}, we get for all $\delta>0$,
\[
    \log N(\delta,\relu,\|\cdot\|_\infty) \le  \gamma(L,\mathbf{p},s)-(s+1) \log(\delta).\label{eq:ent1}
\]
Using the simple bound $\sqrt{a+b}\le \sqrt{a}+\sqrt{b}$, the bracketing integral is upper bounded by
\begin{align*}
J(\delta, \relu, \|\cdot\|_{\infty})&= \int_0^\delta \sqrt{1+\gamma(L,\mathbf{p},s)-(s+1)\log(\ep)}d\ep\\ &\le \int_0^\delta \big(1 + \sqrt{\gamma(L,\mathbf{p},s)}\big)d\ep + \sqrt{s+1} \int_0^1 \sqrt{-\log(\ep)} d\ep \\
&= \delta + \sqrt{\gamma(L,\mathbf{p},s)}\delta +\frac{\sqrt{(s+1)\pi}}{2}\le c\cdot  \gamma(L,\mathbf{p},s)^{1/2}(1+\delta),
\end{align*}
where $c \ge 1$ is some universal constant.

By Lemma \ref{thm:beta}, we have with some constants $K_1,K_2 > 0$ only depending on characteristics of $X_1$, 
\[
    \IE^{*} \sup_{f \in \relu}\big|(\hat \IP_n^X - \IP^X)f\big| \le r_n,
\]
where
\begin{eqnarray*}
    r_n &:=& K_1\cdot n^{-1/2}J_{[]}(F,\relu,\|\cdot\|_{\infty}) + K_2 F\cdot \Big(\frac{1 \vee N_{[]}(2BF,\relu,\|\cdot\|_{\infty})}{n}\Big)^{\frac{\alpha}{\alpha+1}}\\
    &\le& K_1\cdot n^{-1/2}J(\frac{F}{2},\relu,\|\cdot\|_{\infty}) + K_2 F\cdot \Big(\frac{1 \vee N(BF,\relu,\|\cdot\|_{\infty})}{n}\Big)^{\frac{\alpha}{\alpha+1}}\\
    &\le& C\cdot \Big(\big(\frac{\gamma(L,\mathbf{p},s)}{n}\big)^{1/2} + \big(\frac{\gamma(L,\mathbf{p},s)}{n}\big)^{\frac{\alpha}{\alpha+1}}\Big), 
\end{eqnarray*}
and $C > 0$ depends on $F,B,K_1,K_2$. Since $\gamma(L,\mathbf{p},s) \le n$ and $\alpha > 1$, the second summand is dominated by the first. This yields \reff{lem:ent_eq1}.

Note that $\relu(L,\mathbf{p},s,F)$ is separable in $\{f:[0,1]^d \to \IR \text{ meas.}, \|f\|_{\infty} \le F\}$, therefore $\sup_{f\in \relu(L,\mathbf{p},s)}\big|(\hat \IP_n^X - \IP^X)f\big|$ is measurable and
\[
    \IP\Big(\sup_{f\in \relu(L,\mathbf{p},s)}\Big|(\hat \IP_n^X - \IP^X)f\Big| > x\Big) \le \sup_{S\subset \relu(L,\mathbf{p},s) \text{ countable}}\IP\Big(\sup_{f\in S}\Big|(\hat \IP_n^X - \IP^X)f\Big| > x\Big).
\]
By Lemma \ref{lem_talagrand_mixing}, there exists some constant $C_2 > 0$ depending on $F,B,\kappa,\alpha$ such that
\[
    \IP\Big(\sup_{f\in \relu(L,\mathbf{p},s)}\Big|(\hat \IP_n^X - \IP^X)f\Big| \ge C_2\cdot \big(r_n + (\frac{x}{n})^{1/2} + \frac{x}{n}\cdot z^{-\frac{1}{\alpha+1}}\big)\Big) \le 2\exp(-x) + \frac{nz}{x}.
\]
With $x = \log(n)$ and $z = \big(\frac{\log(n)}{n}\big)^{\frac{\alpha+1}{2}}$, we obtain
\[
    \IP\Big(\sup_{f\in \relu(L,\mathbf{p},s)}\Big|(\hat \IP_n^X - \IP^X)f\Big| \ge C_2\cdot \big(r_n + 2(\frac{\log(n)}{n})^{1/2}\big)\Big) \le \frac{2}{n} + \Big(\frac{\log(n)}{n}\Big)^{\frac{\alpha-1}{2}},
\]
which yields \reff{lem:ent_eq2}.
\end{proof}

\begin{lem}[Upper bound on the estimation error]\label{lem_upperbound_estimation_wgan}
    Suppose that there exist constants $\kappa > 1, \alpha > 1$ such that for all $k\in\IN$, $\beta_X(k) \le \kappa\cdot k^{-\alpha}$. Suppose that $\gamma(L_f,\mathbf{p}_f,s_f) \le n$ and $\gamma(L_g\vee L_f,\mathbf{p}_g \vee \mathbf{p}_f,s_g\vee s_f) \le n\Epo$. Then there exists some constant $C > 0$ only depending on characteristics of $X_1$ and $F,\kappa,\alpha$ such that
    \[
        \IE E_n \le C\cdot\Big[ \Big(\frac{\gamma(L_f,\mathbf{p}_f,s_f)}{n}\Big)^{1/2} + \Big(\frac{\gamma(L_g\vee L_f,\mathbf{p}_g \vee \mathbf{p}_f,s_g\vee s_f)}{n\Epo}\Big)^{1/2}\Big].
    \]
    Furthermore, with probability at least $1 - 4n^{-1} - 2(\frac{\log(n)}{n})^{\frac{\alpha-1}{2}}$,
    \[
        E_n \le C\cdot\Big[ \Big(\frac{\gamma(L_f,\mathbf{p}_f,s_f)}{n}\Big)^{1/2} + \Big(\frac{\gamma(L_g\vee L_f,\mathbf{p}_g \vee \mathbf{p}_f,s_g\vee s_f)}{n\Epo}\Big)^{1/2} + \Big(\frac{\log(n)}{n}\Big)^{1/2}\Big].
    \]
\end{lem}
\begin{proof}[Proof of Lemma \ref{lem_upperbound_estimation_wgan}]
    With $g \in \mathcal{R}_G$, $f\in \mathcal{R}_D$, we have
    \[
        f\circ g \in \relu := \relu(L_g+L_f+1, (d_z,p_{g1},\dots,p_{gL_g},d, p_{f1},\dots,p_{fL_f},1), s_g+s_f).
    \]
    Note that there exists some universal constant $c > 0$ such that
    \begin{eqnarray*}
        &&\gamma(L_g+L_f+1,(d_z,p_{g1},\dots,p_{gL_g},d, p_{f1},\dots,p_{fL_f},1),s_g+s_f)\\
        &\le& c\cdot \gamma(L_f\vee L_g, \mathbf{p}_f \vee \mathbf{p}_g, s_f \vee s_g),
    \end{eqnarray*}
    where $x\vee y$ of vectors $x,y$ is meant component-wise.
    
    We now apply Lemma \ref{lem:ent} to both summands of $E_n$. The first summand reads
    \[
        \sup_{f\in \mathcal{R}_D}\big|(\hat \IP_n^X - \IP^X)f\big|
    \]
    with $\beta$-mixing $X_i$. The second summand of $E_n$ is upper bounded by
    \[
        \sup_{h \in \relu}\big|(\hat \IP_{n\Epo}^Z - \IP^Z)h\big|
    \]
    with i.i.d. $Z_i$, that is, $\beta$-mixing coefficients $\beta_Z(k) = \Ii_{\{k = 0\}}$ ($k\ge 0$).
\end{proof}

\begin{proof}[Proof of Theorem \ref{theorem_wgan_excessbayes}]
    By Proposition \ref{basic_inequality_wgan},
    \[
        R_n(\hat g_n) \le \sqrt{d}\cdot A_n + 2\cdot E_n.
    \]
    Under the given assumptions on $L_f, \mathbf{p}_f, s_f$, we conclude from \reff{definition_groesse_netzwerk} (cf. also Remark 1 in \cite{sh}) that 
    \begin{eqnarray*}
        &&\gamma(L_f,\mathbf{p}_f,s_f)\\
        &\le& 2(s_f+1)\log\big(2^{L_f+3}(L_f+1)p_0 p_{L+1} s_f^{L_f}\big) \lesssim s_f L_f \log(s_f L_f).
    \end{eqnarray*}
    
    Under the given assumptions on $L_g, \mathbf{p}_g,s_g$, we conclude by Lemma \ref{lem:approx} for $N$ large enough that
    \[
        A_n \lesssim \frac{N}{n\Epo} + N^{-\beta/d_g}.
    \]
    Thus by Lemma \ref{lem_upperbound_estimation_wgan},
    \begin{eqnarray*}
        \IE R_n(\hat g_n) &\lesssim& \Big(\frac{s_f L_f \log(s_f L_f)}{n}\Big)^{1/2} + N^{-\beta/d_g}\\
        &&\quad\quad + \Big(\frac{(s_f\vee s_g) (L_f\vee L_g) \log((s_f \vee s_g) (L_f \vee L_g))}{n\Epo}\Big)^{1/2}.
    \end{eqnarray*}
    
    Choose $N = \lceil C_1 n\Epo\phi_n\rceil$, where $C_1$ is large enough such that $N \ge (\beta+1)^{d_g} \vee (K+1)e^{d_g}$, then
    \[
         \IE R_n(\hat g_n) \lesssim \Big(\frac{s_f L_f \log(s_f L_f)}{n}\Big)^{1/2} + \phi_n^{1/2}\log(n\Epo)^{3/2}.
    \]
    The large deviation statement is immediate from Lemma \ref{lem_upperbound_estimation_wgan}.
\end{proof}

\subsection{Adaptation to the conditional case}
We follow the same procedure as in the unconditional case with slight adaptations.
We abbreviate $\gmc:=\ssG^c(d_Z,d_Y,d_g,\beta,K)$ and $\rgt := \mathcal{R}(L_g,\mathbf{p}_g,s_g)$, $\rft = \mathcal{R}(L_f,\mathbf{p}_f,s_f)$ as before. Recall from \reff{wgan_conditional} that
\[
    \hat g_{n}^c := \argmin_{g\in \mathcal{R}_G}\hat W_{1,n}^c(g).
\]
\begin{prop}[cWGAN: Basic inequality]\label{basic_inequality_cwgan}
    It holds that
    \[
        \wnc(\hat g_n^c) - \inf_{g\in\gmc} \wnc(g) \le \sqrt{d}\cdot A^c_n + 2\cdot E^c_n,
    \]
    where
    \begin{eqnarray}
        A^c_n &:=& \sup_{g \in \gmc}\inf_{\tilde g\in \rg}\|g - \tilde g\|_{\infty},\label{cwgan_approx_error}\\
        E^c_n &:=& \sup_{f\in \rf} |(\hat\IP_n^{X,Y} - \IP^{X,Y}) f|\nonumber\\
        &&\quad\quad +  \sup_{g \in \rgt,f\in \rft} \Big|\frac{1}{n}\sum_{i=1}^{n}\{f(g(Z_i,Y_i),Y_i) - \IE f(g(Z_1,Y_1),Y_1)\}\Big|.\nonumber\\
        &&\label{cwgan_est_error}
    \end{eqnarray}
\end{prop}
\begin{proof}[Proof of Proposition \ref{basic_inequality_cwgan}]
Proceed as in the proof of \cref{basic_inequality_wgan}. Note that \begin{eqnarray*}
    &&|W^c_{1,n}(g) - W^c_{1,n}(\tilde g)|\\
    &\le& \sup_{f \in \rft, \|f\|_L\le 1} \IE|f(g(Z,Y),Y) - f(\tilde g(Z,Y),Y)|\\
    &\le& \sqrt{d}\|g - \tilde g\|_{\infty},
\end{eqnarray*}
stays the same.
 \end{proof}
 
\begin{lem}\label{lem:c_approx} Let $\beta \ge 1$, $d_g \in\IN$, $\tilde \beta\geq \frac{D}{d_g} \beta$.
    Suppose that for $N$ large enough,
    \begin{itemize}
        \item $L_g\ge \log_2(n) \left(2 \log_2(4d_g\vee 4\beta)+ \log_2(4D\vee 4\tilde\beta) \right),$
        \item $\min_{i=1,...,L_g}p_i\gtrsim N,$
        \item $s_g\gtrsim N \log_2(n)$,
    \end{itemize}
    then $A^c_n$ from \reff{cwgan_approx_error} satisfies
    \begin{equation}
        A^c_n \lesssim \frac{N}{n} + N^{-\beta/d_g},\label{lem:c_approx_eq1}
    \end{equation}
    Here, the bounding constants only depend on $\tilde \beta,\beta,d_g,D,d$ and $K$.
\end{lem}
 \begin{proof}[Proof of Lemma \ref{lem:c_approx}] The proof basically follows from Theorem \ref{reluapprox} with $k = \lceil \log_2(n)\rceil$ along the same lines as in the proof of Theorem 1 in \cite{sh}. Let $g\in \ssG^c = \ssG^c(d_Z,d_Y,d_g,\beta,K)$. For ease of notation, let
\[
    \bm{\beta} = (\beta_0,\beta_1,\beta_2) := (\beta,\tilde \beta,\beta),\quad\quad \mathbf{t} = (t_0,t_1,t_2) := (d_g,D,d_g)     
\]
and $\mathbf{d} = (d_0,d_1,d_2,d_3) := (d_Z+d_Y,D,d_g,d)$. We furthermore abbreviate $g_0 := g_{enc,0}$, $g_1 := g_{enc,1}$ and $g_2 := g_{dec}$.
 
First transform the component functions $g_0,g_1$ as in \cite[Proof of Theorem 1]{sh} to map to $[0,1]^{d_1}$ and $[0,1]^{d_2}$ respectively. Now by Theorem \ref{reluapprox}, for $i = 0,1,2$ we find functions
\begin{eqnarray*}
    \tilde g_{i} &\in& \relu(L_{i},(d_i,\mathbf{p}_{i},d_{i+1}),d_{i+1}s_{i}),
\end{eqnarray*}
where 
\begin{eqnarray*}
    L_{i} &=& 8 + (k+5)(1 + \log_2(t_i \vee \beta_i)), \\
    p_{i} &=& (d_i,6d_{i+1}(t_i+\lceil \beta_i\rceil)N,\dots,6d_{i+1}(t_i+\lceil \beta_i\rceil)N,d_{i+1}) \in \IR^{L_i+2},\\
    s_{i} &\le& 141(t_i+\beta_i+1)^{3+t_i}N(k+6)
    \end{eqnarray*}
such that 
\begin{eqnarray*}
    \|g_{i} - \tilde g_{i}\|_{\infty} &\le& (2K+1)(1+t_i^2+\beta_i^2)6^{t_i}N2^{-k} + K3^{\beta_i} N^{\frac{\beta_i}{t_i}}.
\end{eqnarray*}
Apply $1-(1-\tilde g_{ij})_+$ for $i\in\{0,1\}$ so that the network outputs lie in $[0,1]^{d_{i+1}}$. This does not increase the distance to $g_i$ and adds 4 non-zero parameters per output dimension and 2 layers. The composed network $\tilde g := \tilde g_{2} \circ \sigma(\tilde g_{1}) \circ \sigma(\tilde g_{0})$ satisfies
\begin{eqnarray*}
    \tilde g \in \relu(\bar L, \bar p, \bar s)
\end{eqnarray*}
with
 \begin{eqnarray*}
    \bar L &:=& \sum_{i=0}^{2}L_i + 6,\\
    \bar p &:=& (d_Z+d_Y,p_{0},\dots,p_0,D,p_{1},\dots,p_1,d_g,p_{2},\dots,p_2,d),\\
    \bar s &:=& \sum_{i=0}^{2}d_{i+1}(s_i + 4).
 \end{eqnarray*}
and, in analogy to \cite[Section 7.1, Lemma 3]{sh},
\begin{eqnarray} \label{approx_error_eq}
    \|\tilde g - g\|_{\infty} &\le& C \max_{i=0,1,2}\big\{\frac{N}{n} + N^{-\frac{\beta_i}{t_i}}\big\} = C\big\{\frac{N}{n} +  N^{-\frac{\beta}{d_g}}\big\}
\end{eqnarray}
for a constant $C$ that only depends on $\bm{\beta},\mathbf{d},K$. Up to now, $\tilde g$ may not satisfy $\|\tilde g\|_{\infty} \le F$. However, $\tilde g^{\circ} := (\frac{\|g\|_{\infty}}{\|\tilde g\|_{\infty}}\wedge 1)\tilde g$ still fulfills $\tilde g^{\circ} \in \relu(\bar L, \bar p, \bar s)$ and $\|\tilde g^{\circ}\|_{\infty} \le \|g\|_{\infty} \le K \le F$. Due to $\|\tilde g^{\circ} - g\|_{\infty} \le 2\|\tilde g - g\|_{\infty}$, \reff{approx_error_eq} still holds for $\tilde g^{\circ}$ with changed constants.
\end{proof}

We now provide an analogeous result for Lemma \ref{lem_upperbound_estimation_wgan} in the conditional case. If $Z_i$, $i\in\IZ$ is a sequence of independent random variables and independent of $(X_i,Y_i)$, $i\in\IZ$, then $\beta$-mixing of $(X_i,Y_i)$ implies $\beta$-mixing of $(Y_i,Z_i)$ with the same coefficients. The basic change is that the supremum in the second summand in $E_n^c$ from \reff{cwgan_est_error} runs over a different class of neural networks, namely
\[
\fg = \Big\{ h:[0,1]^{d_Z+d_Y}\rightarrow \R,\; (z,y)\mapsto f(g(z,y),y) \Big| \;g\in\rg,\; f\in\rf,\; \|f\|_L\le 1 \Big\}.
\]
By adding $d_Y$ neurons in each layer of $g$, we can mimic the function $(g(z,y),y)$. Thus
\[
    \fg\subseteq \tilde{\relu}^c=\relu(L_{comp},\mathbf{p}_{comp}, s_{comp},\infty),
\]
where
\begin{eqnarray*}
    L_{comp} &=& L_f+L_g+1,\\
    \mathbf{p}_{comp} &=& (d_z +d_Y,p_{g,1} + d_Y,\dots,p_{g,L_g} + d_Y,d+d_Y,p_{f,1},\dots,p_{f,L_f},1),\\
    s_{comp} &:=& s_g + s_f + (L_g+1)d_Y.
\end{eqnarray*}
As long as $s_g \ge d_Y L_g$, $p_{g,i} \ge d_Y$ ($i=1,...,L_g$), there exists a universal constant $c > 0$ such that
\[
    \gamma(L_{comp},\mathbf{p}_{comp},s_{comp}) \le c\cdot \gamma(L_f \vee L_g, \mathbf{p}_{f} \vee \mathbf{p}_g, s_f \vee s_g),
\]
where $x \vee y$ for vectors $x,y$ is meant component-wise. These remarks lead to the following result.

\begin{lem}[Upper bound on the estimation error]\label{lem_upperbound_estimation_wgan_conditional}
    Suppose that there exist constants $\kappa > 1, \alpha > 1$ such that for all $k\in\IN$, $\beta_{X,Y}(k) \le \kappa\cdot k^{-\alpha}$. Suppose that $s_g \ge d_Y L_g$, $p_{g,i} \ge d_Y$ ($i=1,...,L_g$) and $\gamma(L_g\vee L_f,\mathbf{p}_g \vee \mathbf{p}_f,s_g\vee s_f) \le n$. Then there exists some constant $C > 0$ only depending on characteristics of $(X_1,Y_1)$ and $F,\kappa,\alpha$ such that 
    \[
        \IE E_n^c \le C\cdot\Big[ \Big(\frac{\gamma(L_f,\mathbf{p}_f,s_f)}{n}\Big)^{1/2} + \Big(\frac{\gamma(L_g\vee L_f,\mathbf{p}_g \vee \mathbf{p}_f,s_g\vee s_f)}{n}\Big)^{1/2}\Big].
    \]
    Furthermore, with probability at least $1 - 4n^{-1} - 2(\frac{\log(n)}{n})^{\frac{\alpha-1}{2}}$,
    \[
        E_n \le C\cdot\Big[ \Big(\frac{\gamma(L_f,\mathbf{p}_f,s_f)}{n}\Big)^{1/2} + \Big(\frac{\gamma(L_g\vee L_f,\mathbf{p}_g \vee \mathbf{p}_f,s_g\vee s_f)}{n}\Big)^{1/2} + \Big(\frac{\log(n)}{n}\Big)^{1/2}\Big].
    \]
\end{lem}

\begin{proof}[Proof of Theorem \ref{theorem_wgan_conditional_excessbayes}]
In order to bound the estimation error $E_n^c$ proceed as in the proof of Theorem \ref{theorem_wgan_excessbayes}. By Proposition \ref{basic_inequality_cwgan},
    \[
        R_n^c(\hat g_n^c) \le \sqrt{d}\cdot A_n^c + 2\cdot E_n^c.
    \]
    Under the given assumptions on $L_f, \mathbf{p}_f, s_f$, we conclude from \reff{definition_groesse_netzwerk} that 
    \begin{eqnarray*}
        &&\gamma(L_f,\mathbf{p}_f,s_f)\\
        &\le& (s_f+1)\log\big(2^{2L_f+6}(L_f+1)p_0^{2}p_{L+1}^2s_f^{2L_f}\big) \lesssim s_f L_f \log(s_f L_f).
    \end{eqnarray*}

    Under the given assumptions on $L_g, \mathbf{p}_g,s_g$, we conclude by Lemma \ref{lem:c_approx} for $N$ large enough that
    \[
        A_n \lesssim \frac{N}{n} + N^{-\beta/d_g}.
    \]
    Thus by Lemma \ref{lem_upperbound_estimation_wgan_conditional},
    \begin{eqnarray*}
        \IE R_n^c(\hat g_n^c) &\lesssim& \Big(\frac{s_f L_f \log(s_f L_f)}{n}\Big)^{1/2} + N^{-\beta/d_g}\\
        &&\quad\quad + \Big(\frac{(s_f\vee s_g) (L_f\vee L_g) \log((s_f \vee s_g) (L_f \vee L_g))}{n}\Big)^{1/2}.
    \end{eqnarray*}
    
    Choose $N = \lceil C_1 n\phi_n\rceil$, where $C_1$ is large enough such that the conditions of Lemma \ref{lem:c_approx} are met. Then
    \[
         \IE R_n^c(\hat g_n^c) \lesssim \Big(\frac{s_f L_f \log(s_f L_f)}{n}\Big)^{1/2} + \phi_n^{1/2}\log(n)^{3/2}.
    \]
\end{proof}

\section{Entropy bound and large deviation bounds for absolutely regular sequences}\label{sec:entropybound}

\subsection{Entropy bounds}
In this section, we develop the entropy bound under absolutely regular $\beta$-mixing for deep sparse regularized ReLU networks $\relu=\relu(L,\mathbf{p},s,F)$ in dependence on $L,\mathbf{p}$ and $s$. The basic theoretical ingredients consist of the empirical process theory invented in \cite{betanorm} and \cite{dl}. Recall the introduction of bracketing numbers and mixing coefficients from Section \ref{sec:apprerror}.

%$(\mathcal{L}_{2,\beta}(\P^X),||\cdot||_{2,\beta})$ is a normed subspace of $(\mathcal{L}_2(\P^X),||\cdot||_2)$ with $||f||_2\le ||f||_{2,\beta}\le \sqrt{\sum \beta(k)} \cdot ||f||_\infty$, where the last inequality follows since $Q_f$ is bounded by $||f||_\infty$.\\

%The following two lemmas formulate the key results our entropy bound uses.

%By the following theorem, we do not lose the entropy rate from the i.i.d. case in expectation, which by \cite[19.35]{vdvas} would consist of the $J_{[]}$-term.

\begin{lem}\label{thm:beta}
Let $\ssF \subset \{f:\IR^r \to \IR \text{ measurable}\}$ be any class of functions such that $\sup_{f\in \ssF}\|f\|_{\infty} \le F$. Let
\[
    H := 1\vee \log N_{[]}(2BF, \ssF, \|\cdot\|_{\infty}).
\]
Suppose that there exist constants $\kappa > 1, \alpha > 1$ such that for all $k\in\IN$, $\beta_X(k) \le \kappa\cdot k^{-\alpha}$ and $H \le n$.

Then there exist constants $K_1,K_2 > 0$ only depending on characteristics of $(X_i)_{i\in\IZ}$ such that
\begin{eqnarray}
    &&\IE^* \sup_{f\in\ssF} |(\hat \IP^X_n-\IP^X)f|\nonumber\\
    &\le& K_1 \cdot n^{-1/2}\cdot J_{[]}(F,\ssF,\|\cdot\|_\infty) + K_2\, F \cdot \Big( \frac{H}{n}\Big)^{\frac{\alpha}{\alpha+1}} =: r_n,\label{thm:beta_eq0}
\end{eqnarray}
where $\IE^{*}$ denotes the outer expectation.
\end{lem}
\begin{proof}
Let $\delta>0$ arbitrary. From \cite{dl} (Remark 3.7 and the procedure in section 4.3 therein) yield that for any class $\ssF$ with $\sup_{f\in \ssF}\|f\|_{2,\beta} \le \delta$, $\sup_{f\in \ssF}\|f\|_{\infty} \le F$, we have with some constant $K' > 0$ only depending on characteristics of $(X_i)_{i\in\IZ}$ that
\begin{equation}
    \IE^* \sup_{f\in\ssF} |(\hat \IP^X_n-\IP^X)f|\le K' \cdot n^{-1/2}\cdot J_{[]}(\delta,\ssF,\|\cdot\|_{2,\beta}) + 2F \Ii_{2F>M(n,\delta)} + 2\, R(n,\delta),\label{thm:beta_eq1}
\end{equation}
where with $q^{*}(x):= \min\{q\in\N:\; \beta_X(q)\le q x \}$ and some universal constant $c > 0$,
\begin{eqnarray*}
    R(n,\delta)&:=& 2c F \frac{H(\delta)\cdot  q^{*}\Big(\frac{H}{n}\Big)}{n},\\
    M(n,\delta) &:=& \frac{16B^{1/2}\delta}{q^{*}( H(\delta)/n)} \sqrt{\frac{n}{H(\delta)}},\\
    H(\delta) &:=& 1 \vee \log N_{[]}(2B^{1/2}\delta, \ssF, \|\cdot\|_{2,\beta}).
\end{eqnarray*}
By \reff{standardbound_betnorm}, choosing $\delta = B^{1/2} \cdot F$ yields $H(\delta)\le H$ and $J_{[]}(\delta, \ssF, \|\cdot\|_{2,\beta}) \le B^{1/2} J_{[]}(F,\ssF,\|\cdot\|_{\infty})$.

For $x \le 1$, we have
\begin{align}
    q^{*}(x)=\min\{ q\in\N: \, \beta(q)\le q x \} &\le \min\{ q\in\N: \kappa q^{-\alpha} \le qx \}\nonumber\\
    &= \min\{ q\in\N: \kappa x^{-1} \le q^{\alpha+1} \} = \lceil \kappa^{\frac{1}{\alpha+1}} x^{-\frac{1}{\alpha+1}}\rceil\nonumber\\
    &\le 2\kappa^{\frac{1}{\alpha+1}}x^{-\frac{1}{\alpha+1}}.\label{thm:beta_eq10}
\end{align}
Due to $H \le n$, this implies
\begin{equation}
R(n,\delta) \le 2cF\cdot \frac{H}{n} \cdot q^{*}\Big(\frac{H}{n} \Big)\le  4cF\kappa^{\frac{1}{\alpha+1}}\Big(\frac{H}{n}\Big)^{\frac{\alpha}{\alpha+1}}.\label{thm:beta_eq2}
\end{equation}
In the same way we get that
\[
 M(n,\delta) = 16BF\Big[\sqrt{\frac{H}{n}} \cdot q^{*}\Big(\frac{H}{n} \Big)\Big]^{-1}\ge 8BF\delta\kappa^{-\frac{1}{\alpha+1}} \Big( \frac{H}{n}\Big)^{\frac{1}{\alpha+1}-\frac{1}{2}}.
\]
For any $p > 0$, we have
\[
    2F \Ii_{2F>M(n,\delta)} \le (2F)^{1+p}M(n,\delta)^{-p} \le (2F)^{1+p}\cdot (8BF)^{-p}\kappa^{\frac{p}{\alpha+1}}\Big(\frac{H}{n}\Big)^{-\frac{p(\alpha-1)}{2(\alpha+1)}}.
\]
Choosing $p = \frac{2\alpha}{\alpha-1}$ yields
\begin{equation}
    2F \Ii_{2F>M(n,\delta)} \le (2F)^{1+p}\cdot (8BF)^{-p}\kappa^{\frac{p}{\alpha+1}}\Big(\frac{H}{n}\Big)^{-\frac{\alpha}{\alpha+1}}.\label{thm:beta_eq3}
\end{equation}
Insertion of \reff{thm:beta_eq2} and \reff{thm:beta_eq3} into \reff{thm:beta_eq1} yields the result.

%Finally, since $||f||_{2,\beta}\le \sqrt{\sum \beta(k)}\cdot ||f||_\infty$, we get \begin{equation}
%    N_{[]}(\sqrt{\sum \beta(k)}\cdot \delta, \mathcal{F}_n, ||\cdot||_{2,\beta})\le N_{[]} (\delta,\mathcal{F}_n,||\cdot||_\infty)\label{eq:n2betatonsup}
%    \end{equation}
%and thus \begin{align*}
%    \int_0^\delta \sqrt{1\vee \log N_{[]}(u,\mathcal{F}_n,||\cdot||_{2,\beta}) } du &\le \int_0^\delta \sqrt{1+ \log N_{[]}\left(\frac{u}{\sqrt{\sum \beta(k)}},\mathcal{F}_n,||\cdot||_\infty \right) } du\\
%    &= \sqrt{\sum \beta(k)} \cdot \int_0^{\delta/\sqrt{\sum \beta(k)}}  \sqrt{1+ \log N_{[]}\left(v,\mathcal{F}_n,||\cdot||_\infty \right) } dv\\
%    &= \sqrt{\sum \beta(k)} \cdot J_{[]} (F,\mathcal{F}_n,||\cdot||_\infty),
%\end{align*}
%which finishes the proof.
\end{proof}

\subsection{Large deviation bounds}

The essential techniques we use to derive large deviations bounds under absolutely regular $\beta$-mixing are coupling (cf. \cite{coupling, dl}), a Talagrand-type concentration inequality by  \cite{klein2005} and a covariance bound by Rio \cite{rio_cov}. For completeness, we cite the results our derivations are based on.

\begin{lem}[Coupling lemma, \cite{dl}] \label{lem:coupling}
Let $X$ and $Y$ be two random variables taking values in the Borel spaces $\mathcal{X}_1$ and $\mathcal{X}_2$ respectively, and let $U$ be a random variable with uniform distribution on $[0,1]$, independent of $(X,Y)$. There exists a random variable $Y^* = h(X,Y,U)$, where $h$ is a measurable function from $\mathcal{X}_1\times\mathcal{X}_2\times [0,1]$ into $\mathcal{X}_2$, such that:
\begin{enumerate}[(i)]
    \item $Y^*$ is independent of $X$ and has the same distribution as $Y$.
    \item $\P(Y\neq Y^*)=\beta(\sigma(X), \sigma(Y)).$
\end{enumerate}
\end{lem}
\begin{thm}[Talagrand-type concentration inequality, \cite{klein2005}, Theorem 1.1]\label{bennett_bousquet}
Assume the $W_i, i\in \IN$ are independent random variables with values in $\IR^r$. Let $\ssF \subset \{f:\IR^{r} \to \IR \text{ meausurable}\}$ be a countable set of functions with $\IE f(W_1) = 0$, $\sup_{f\in \ssF}\| f(W_1)\|_2^2 < \infty$ and $\sup_{f\in \ssF}\|f\|_{\infty} \le 1$. Define
\[
Z=\sup_{f\in\ssF} \Big| \sum_{i=1}^{m} f(W_i)\Big|.
\]
Let $\sigma$ be a positive real number such that $\sigma^2\geq \sup_{f\in\ssF} \mathrm{Var}\|f(W_1)\|_2^2$. Then for all $x > 0$, it holds that
\[
\P\left( Z\geq \IE Z + \left( 2 x (m\sigma^2 + 2 \mathds{E} Z) \right)^{1/2} + \frac{x}{3}\right)\le \exp(-x).
\]
\end{thm}

The following lemma is a direct consequence of Corollary 1.4 and Remark 1.6 in \cite{rio_cov}.
\begin{lem}[Variance bound for $\beta$-mixing sequences]\label{beta_var_bound}
Let $(X_i)_{i\in\N}$ be a strictly stationary sequence of random variables with values in a Polish space $\mathcal{X}$. Let $q\in\N$. Then for any $f:\mathcal{X} \to \IR$ with $\|f\|_{2,\beta} < \infty$,
\[
    \Big\|\sum_{i=1}^q f(X_i)\Big\|_2^2 \le 4 q  \|f\|^2_{2,\beta}.
\]
\end{lem}

\begin{lem}\label{lem_talagrand_mixing}
    Let $\ssF \subset \{f:\IR^r \to \IR \text{ measurable}\}$ be any countable class of functions such that $\IE f(X_i) = 0$ for all $f\in \ssF$ and $\sup_{f\in \ssF}\|f\|_{\infty} \le F$.
    
    Suppose that there exist constants $\kappa > 1, \alpha > 1$ such that for all $k\in\IN$, $\beta_X(k) \le \kappa\cdot k^{-\alpha}$. Define $r_n$ as in \reff{thm:beta_eq0}.
    
    Then for all $x > 0$ and $q \in\IN$,
    \begin{equation}
        \IP\Big(\sup_{f\in \ssF}\Big|\sum_{i=1}^{n}f(X_i)\Big| \ge 4nr_n + 8B^{1/2}Fn^{1/2} x^{1/2} + 4Fqx\Big) \le 2\exp(-x) + \frac{n\beta_X(q)}{qx}.\label{lem_talagrand_mixing_res1}
    \end{equation}
    Especially, for any $z \le 1$, 
    \begin{equation}
        \IP\Big(\sup_{f\in \ssF}\Big|\sum_{i=1}^{n}f(X_i)\Big| \ge 4nr_n + 8B^{1/2}Fn^{1/2} x^{1/2} + 8F\kappa^{\frac{1}{\alpha+1}}z^{-\frac{1}{\alpha+1}}x\Big) \le 2\exp(-x) + \frac{nz}{x}.\label{lem_talagrand_mixing_res2}
    \end{equation}
\end{lem}
\begin{proof}[Proof of Lemma \ref{lem_talagrand_mixing}]
    Let $q\in\N$. Starting from \Cref{lem:coupling} we construct by induction a sequence of random variables $(X_i^0)_{i>0}$ such that:
\begin{enumerate}
    \item For any $i\geq 0$, the random variable $U_i^0:=(X^0_{iq+1},\dots,X^0_{iq+q})$ has the same distribution as $U_i:=(X_{iq+1},\dots,X_{iq+q})$.
    \item The sequence $(U_{2i}^0)_{i\geq 0}$ is i.i.d. and so is $(U^0_{2i+1})_{i\geq0}$.
    \item For any $i\geq 0$, $\P(U_i\neq U_i^0)\le \beta(q)$.
\end{enumerate}
Define $W_i(f) := \sum_{j=(i-1)q+1}^{iq \wedge n}f(X_i^{0})$. From the above coupling we obtain the following decomposition:
\[
    \sum_{i=1}^{n}f(X_i) = \sum_{i=1, i \text{ odd}}^{\lceil \frac{n}{q}\rceil}W_i(f) + \sum_{i=1, i \text{ even}}^{\lceil \frac{n}{q}\rceil}W_i(f) + \sum_{i=1}^{n}\big\{f(X_i) - f(X_i^0)\big\}.
\]
We conclude that
\begin{equation}
    \sup_{f\in \ssF}\Big| \sum_{i=1}^{n}f(X_i)\Big| \le Fq\cdot (Z_1 + Z_2) + A,\label{lem_talagrand_mixing_eq0}
\end{equation}
where
\[
    Z_1 := \sup_{f\in \ssF}\Big|\sum_{i=1, i \text{ odd}}^{\lceil \frac{n}{q}\rceil}\frac{W_i(f)}{Fq}\Big|, \quad\quad Z_2 := \sup_{f\in \ssF}\Big|\sum_{i=1, i \text{ even}}^{\lceil \frac{n}{q}\rceil}\frac{W_i(f)}{Fq}\Big|, \quad\quad A := F\sum_{i=1}^{n}\Ii_{X_i \not= X_i^{0}}.
\]
Note that $W_{2k}(f)$, $W_{2k+1}(f)$, $k \ge 0$ are independent by construction. Furthermore, $\sup_{f\in \ssF}|\frac{Y_i(f)}{Fq}| \le 1$ and by Lemma \ref{beta_var_bound} and \reff{standardbound_betnorm},
\[
    \|W_{i}(f)\|_2^2 \le 4q\|f\|_{2,\beta}^2 \le 4qB \|f\|_{\infty}^2 \le 4q BF^2
\]
and thus $\|\frac{W_i(f)}{Fq}\|_2^2 \le \frac{4B}{q}$. By Theorem \ref{bennett_bousquet} applied with $\sigma^2 = \frac{4B}{q}$, we have
\[
    \IP\Big(Z_1 \ge \IE Z_1 + (2x(\frac{4Bn}{q^2} + 2\IE Z_1))^{1/2} + \frac{x}{3}\Big) \le \exp(-x).
\]
Using the simple bound $2ab \le a^2 + b^2$, we have
\[
    (2x(\frac{4Bn}{q^2} + 2\IE Z_1))^{1/2} \le \frac{(8Bxn)^{1/2}}{q} + 2(x\IE Z_1)^{1/2} \le \frac{(8Bxn)^{1/2}}{q} + x + \IE Z_1.
\]
This yields
\begin{equation}
    \IP\Big(Z_1 \ge 2\IE Z_1 + (\frac{8Bxn}{q^2})^{1/2} + \frac{4x}{3}\Big) \le \exp(-x).\label{lem_talagrand_mixing_eq1}
\end{equation}
Let $I_1 := \bigcup_{i=1, \text{$i$ odd}}^{\lceil \frac{n}{q}\rceil}\{(i-1)q+1,...,iq\wedge n\}$. Then $Z_1 = \frac{1}{Fq}\sup_{f\in \ssF}\big|\sum_{i \in I_1}f(X_i^{0})\big|$. Note that by construction, $X_i^{0}$ is still $\beta$-mixing with coefficients upper bounded by $\beta_X$, and $I_1$ has less than or equal $n$ summands. It is therefore easily seen that the upper bounds in \cite{dl} which were used in the proof of Lemma \ref{thm:beta} stay the same. We obtain from \reff{thm:beta_eq0} that 
\begin{equation}
    Fq\cdot \IE Z_1 \le nr_n.\label{lem_talagrand_mixing_eq2}
\end{equation}
Similar results as given in \reff{lem_talagrand_mixing_eq1} and \reff{lem_talagrand_mixing_eq2} also hold for $Z_2$.

Finally, we have by Markov's inequality that
\begin{equation}
    \IP(A > x) \le \frac{\|A\|_1}{x} \le \frac{nF \IP(X_i \not= X_i^0)}{x} \le \frac{nF\beta_X(q)}{x}.\label{lem_talagrand_mixing_eq3}
\end{equation}
Insertion of \reff{lem_talagrand_mixing_eq1}, \reff{lem_talagrand_mixing_eq2} and \reff{lem_talagrand_mixing_eq3} into \reff{lem_talagrand_mixing_eq0} yields
\begin{eqnarray*}
    &&\IP\Big(\sup_{f\in \ssF}\Big|\sum_{i=1}^{n}f(X_i)\Big| \ge 2\cdot\big(2nr_n + (8BF^2nx)^{1/2} + \frac{4Fqx}{3}\big) + Fqx\Big)\\
    &\le& \IP(Z_1 \ge 2\IE Z_1 + (\frac{8Bxn}{q^2})^{1/2} + \frac{4x}{3}) + \IP(Z_2 \ge 2\IE Z_2 + (\frac{8Bxn}{q^2})^{1/2} + \frac{4x}{3})\\
    &&\quad\quad + \IP(A \ge Fqx)\\
    &\le& 2 \exp(-x) + \frac{n\beta_X(q)}{qx},
\end{eqnarray*}
which concludes the proof of \reff{lem_talagrand_mixing_res1}. \reff{lem_talagrand_mixing_res2} follows from the upper bound in \reff{thm:beta_eq10} and the choice $q = q^{*}(z)$.
\end{proof}

\end{document}